\documentclass[11pt]{article}

\usepackage{amsmath}
\usepackage{amssymb}
\usepackage{amsfonts}
\usepackage{amsthm}
\usepackage{array}
\usepackage{bm}
\usepackage[shortlabels]{enumitem}
\usepackage{stmaryrd}
\usepackage{mathtools}

\usepackage{amsxtra}
\usepackage{array}
\usepackage{mathrsfs}
\usepackage{slashed}
\usepackage{xcolor}
\usepackage{tikz-cd}
\usepackage{tkz-euclide} 
\usepackage{pgfplots}

\newcolumntype{C}[1]{>{\centering\hspace{0pt}}p{#1}}
\usepackage[margin=1in]{geometry}

\newcommand{\SO}{\mathrm{SO}}
\newcommand{\Spin}{\mathrm{Spin}}
\newcommand{\U}{\mathrm{U}}
\newcommand{\SU}{\mathrm{SU}}
\newcommand{\Sp}{\mathrm{Sp}}

\newcommand{\Z}{\mathbb{Z}}

\newcommand{\R}{\mathbb{R}}
\newcommand{\C}{\mathbb{C}}

\newcommand{\CP}{\mathbb{CP}}

\newcommand{\vol}{\mathrm{vol}}
\newcommand{\Hol}{\mathrm{Hol}}

\newtheorem{thm}{Theorem}[section]
\newtheorem{prop}[thm]{Proposition}
\newtheorem{lem}[thm]{Lemma}
\newtheorem{cor}[thm]{Corollary}

\theoremstyle{definition}
\newtheorem{defn}[thm]{Definition}
\newtheorem{conv}[thm]{Convention}
\newtheorem{example}[thm]{Example}
\newtheorem{rmk}[thm]{Remark}

\usepackage[nottoc]{tocbibind}
\numberwithin{equation}{section}

\usepackage{hyperref}
\hypersetup{colorlinks,linkcolor=black,citecolor=black,urlcolor=black}

\title{Montel's theorem and tautness \\
in calibrated geometry}
\author{Anton Iliashenko, Spiro Karigiannis, Jesse Madnick}
\date{October 2026}

\newcommand{\Addresses}
{{  \bigskip
\noindent	\textsc{Beijing Institute of Mathematical Sciences and Applications} \par\nopagebreak
\noindent	\textsc{Beijing, China} \par\nopagebreak
\noindent	\texttt{antoniliashenko@bimsa.cn} \\

\medskip
\noindent	\textsc{University of Waterloo} \par\nopagebreak
\noindent	\textsc{Waterloo, ON, Canada} \par\nopagebreak
\noindent	\texttt{karigiannis@uwaterloo.ca} \\

\medskip
\noindent	\textsc{Seton Hall University} \par\nopagebreak
\noindent	\textsc{South Orange, NJ, United States} \par\nopagebreak
\noindent	\texttt{jesse.ochs.madnick@gmail.com} \\
}}

\usepackage{microtype}
\let\originalO\O
\renewcommand{\O}{\ifmmode\text{\originalO}\else\originalO\fi}
\begin{document}

\maketitle
\begin{center}\small
\end{center}

	\begin{abstract} We relate the hyperbolicity of a calibrated manifold $(X, \phi)$ to the analytic properties of the space of Smith immersions $\mathrm{SmIm}(B^k, X)$ from the Poincar\'{e} $k$-ball into $X$.  In particular, we establish the following calibrated analogue of a theorem of Royden: if $X$ is $\phi$-saturated, then $R_\phi$- and $K_\phi$-hyperbolicity coincide, and either implies the equicontinuity of $\mathrm{SmIm}(B^k, X)$ with respect to the $\phi$-distance.  This yields a Montel theorem for compact $\phi$-saturated calibrated manifolds as an immediate corollary.  Our primary technical tool is a new Schwarz lemma for Smith immersions from $B^k$ into $X$, which is of independent interest.  \\
    \indent In a similar spirit, we also prove a calibrated analogue of Kiernan's theorem to the effect that the $K_\phi$-hyperbolicity of $X$ is almost equivalent to $\mathrm{SmIm}(B^k, X)$ being a normal family.  Finally, we prove that bounded domains in flat euclidean space are $R_\phi$-hyperbolic for any calibration $\phi$, and we investigate the hyperbolicity of products and discrete quotients.
\end{abstract}

  \tableofcontents

\section{Introduction}

\indent \indent Montel's theorem is a cornerstone of complex analysis, providing sufficient conditions for a sequence of holomorphic functions to admit a convergent subsequence.  One version of the result states that if $B^2$ is the unit disk and $V \subset \C$ is a bounded domain, then $\Hol(B^2, V) = \{\text{holomorphic } f\colon B^2 \to V\}$ is a pre-compact subset of $C^0(B^2, \C)$ in the compact-open topology.  This result has deep applications in complex analysis --- including the Riemann mapping theorem and Big Picard Theorem --- and is a crucial tool in complex dynamics~\cite{Abate}. \\ 
\indent As befits such an important compactness result, Montel's theorem has been extended in a wide variety of ways.  For example, its analogue for compact Riemann surfaces reveals a striking connection with geometry:

\begin{thm}[Montel for compact Riemann surfaces] \label{thm:MontelBasic}
 Let $X$ be a compact Riemann surface.  Then $X$ is hyperbolic if and only if $\Hol(B^2, X)$ is a pre-compact subset of $C^0(B^2, X)$.  (In fact, it is a compact subset.)
 \end{thm}

\indent Theorem \ref{thm:MontelBasic} has itself been generalized to higher dimensions, and to the non-compact setting.  Notably, Royden~\cite{Royden} proved that a complex manifold $X$, not necessarily compact, is Kobayashi hyperbolic if and only if $\mathrm{Hol}(B^2, X)$ is equicontinuous (for some metric inducing the topology).  In the process, he established the necessity and sufficiency of a simple pointwise condition (``Royden's criterion") for Kobayashi hyperbolicity.  Similarly, Kiernan~\cite{Kiernan} proved that the Kobayashi hyperbolicity of $X$ is almost (but not quite) equivalent to requiring that $\mathrm{Hol}(B^2, X)$ be a normal family.  In addition to yielding Theorem \ref{thm:MontelBasic} as a quick corollary, these two general results showcase deep relationships between the geometry of $X$, on the one hand, and the analytic properties of $\mathrm{Hol}(B^2, X)$, on the other.

\subsection{Main results} \label{sub:MainResults}

\indent \indent In this work, we prove calibrated analogues of Royden's, Kiernan's, and Montel's theorems, thereby linking the hyperbolicity of a calibrated manifold $(X, g_X, \phi)$ to the analytic aspects of $\mathrm{SmIm}(B^k, X)$, the space of Smith immersions from the $k$-ball into $X$.  Since the notion of ``Smith immersion" is relatively new, let us now recall the relevant definitions. \\
\indent Let $(X, g_X, \phi)$ be a calibrated manifold, by which we mean an oriented Riemannian $n$-manifold $(X,g_X)$ together with a calibration $\phi \in \Omega^k(X)$.  A \emph{Smith immersion} (also called a \emph{Smith map} or a \emph{conformal $\phi$-curve}) is a $C^1$ map $f \colon \Sigma^k \to X^n$ whose domain is an oriented Riemannian $k$-manifold $(\Sigma^k, g_\Sigma)$ that satisfies
\begin{align*}
f^* g_X & = \lambda^2 g_\Sigma, \\
f^* \phi & = \lambda^k \vol_\Sigma,
\end{align*}
for some function $\lambda \colon \Sigma \to [0,\infty)$.  The first equation says that $f$ is weakly conformal, which implies that $\mathrm{rank}(df_x) = k$ or $\mathrm{rank}(df_x) = 0$ at each $x \in \Sigma$.  Geometrically, Smith immersions of full rank are precisely the conformal immersions whose images are $\phi$-calibrated submanifolds.  In this way, Smith immersions provide a conformal mapping approach to the study of calibrated geometry.  Quasiconformal generalizations of Smith immersions have been studied by Pankka \cite{Pankka}, Ikonen--Pankka~\cite{IP}, and others, and lead to a fruitful theory.\\
\indent When $\phi = \omega \in \Omega^2(X)$ is a K\"{a}hler calibration, the Smith immersions into $X$ are precisely the holomorphic curves $f \colon \Sigma^2 \to X$ from Riemann surfaces into $X$.  Indeed, Smith immersions enjoy several properties analogous to those of holomorphic curves~\cite{CKM}.  For example, just as holomorphic curves into K\"{a}hler manifolds are harmonic maps, it turns out that Smith immersions into calibrated manifolds are $k$-harmonic maps. \\
\indent Motivated by this analogy, Broder--Iliashenko--Madnick \cite{BIM} used Smith immersions in place of holomorphic curves to define various hyperbolicity notions for calibrated manifolds.  The starting point is the \emph{KR $\phi$-metric} $K_X \colon TX \to [0,\infty]$, defined by
$$K_X(v_p) = \inf\!\left\{ a > 0 \colon \exists f \in \mathrm{SmIm}(B^k, X) \text{ s.t. } f(0) = p, \, df_0(e_1) = \frac{1}{a}v\right\},$$
where the $k$-ball $B^k$ is equipped with the Poincar\'{e} metric.  A calibrated manifold $X$ is called \emph{$R_\phi$-hyperbolic at $p \in X$} if there exists a neighborhood $U \subset X$ and a constant $c > 0$ such that $K_X(v) \geq c|v|$ for all $v \in TU$. We say $X$ is $R_\phi$-hyperbolic if it is so at every $p \in X$. In the case of a K\"{a}hler calibration $\phi = \omega$, the KR $\omega$-metric is precisely the Kobayashi-Royden pseudometric, and $R_\omega$-hyperbolicity is Royden's criterion~\cite{Royden}. \\
\indent We say $X$ is \emph{$\phi$-saturated} if $K_X$ is upper semicontinuous at every point of the zero section of $TX$. We say that $X$ is \emph{$K_\phi$-hyperbolic} if the pseudo-distance $d_\phi \colon X \times X \to [0,\infty]$ defined by
$$d_\phi(p,q) = \inf_\gamma \int_0^1{}^{*} K_X(\gamma'(t))\,dt,$$
where the infimum runs over all piecewise $C^1$ curves $\gamma \colon [0,1] \to X$ from $p$ to $q$, is non-degenerate. Here $\int^{*}$ denotes the upper Lebesgue integral, defined in $\S$\ref{sub:Hyp-Prelim}. Our use of the upper Lebesgue integral follows Ikonen--Pim~\cite{IkonenPim2026}. When $\phi = \omega$ is a K\"{a}hler calibration, $K_X$ is upper semicontinuous on $TX$, so $\omega$-saturation is automatic, and $K_\omega$-hyperbolicity is exactly Kobayashi hyperbolicity.

In $\S$\ref{sec:EquivHyp}, we prove the following calibrated analogue of Royden's theorem.

\begin{thm}[Royden for calibrated manifolds] \label{thm:RoydenMain}
Let $(X, g, \phi)$ be a $\phi$-saturated calibrated manifold. The following are equivalent:
\begin{enumerate}\setlength\itemsep{-0.5mm}
 \item[\emph{(1)}] $X$ is $R_{\phi}$-hyperbolic.
 \item[\emph{(2)}] $X$ is $K_{\phi}$-hyperbolic.
 \item[\emph{(3)}] $X$ is $K_{\phi}$-hyperbolic and $d_{\phi}$ induces the standard topology.
 \item[\emph{(4)}] $X$ is $K_\phi$-hyperbolic and $\mathrm{SmIm}(B^k,X)$ is pointwise equicontinuous with respect to $d_\phi$.
 \item[\emph{(5)}] $\mathrm{SmIm}(B^k,X)$ is an even family.
\end{enumerate}
\end{thm}
\noindent Note that our proof of (5) $\implies$ (1) does not use the $\phi$-saturation assumption.

\begin{cor}[Montel for compact calibrated manifolds] \label{cor:MontelCalibrated} Let $(X, g, \phi)$ be a compact $\phi$-saturated calibrated manifold.  Then $X$ is $K_\phi$-hyperbolic if and only if $\mathrm{SmIm}(B^k, X)$ is pre-compact in $C^0(B^k, X)$.
\end{cor}

\noindent Note that by taking $\dim(X) = 2$ and $\phi = \omega$ to be a K\"{a}hler calibration in Corollary \ref{cor:MontelCalibrated}, we recover Theorem  \ref{thm:MontelBasic} as a special case. \\

The most difficult implication in Theorem \ref{thm:RoydenMain} is (5) $\implies$ (1). To prove this, we establish in $\S$\ref{sec:Schwarz} the following Schwarz lemma for Smith immersions, which is of independent interest.

\begin{thm}[Schwarz lemma] \label{thm:SchwarzMain}
Let $(X,g_X,\phi)$ be calibrated. For every $p\in X$, there exist
$R_0>0$ and $C>0$ with the following property. Let $0<R\leq R_0$.
If there exist radii $\delta\in(0,1)$ and $\delta'\in(0,R)$ such that
all Smith immersions $f\colon B^k\to X$ with $f(0)\in B_{\delta'}(p)$
satisfy $f(B_\delta(0))\subset B_R(p)$, then every such $f$ satisfies
\[
\Vert df_0\Vert\leq C\frac{R}{\delta}.
\]
\end{thm}

Also in $\S$\ref{sec:Schwarz}, we consider the question of which domains in flat euclidean space are $R_\phi$-hyperbolic, for a given calibration $\phi \in \Omega^k(\R^n)$.  In this direction, we prove:

\begin{thm}[Bounded domains are $R_\phi$-hyperbolic] \label{thm:BddDomMain}
Let $(U, g_0, \phi)$ be a calibrated manifold for which $U \subset \R^n$ is an open set, $g_0$ is the flat metric, and $\phi \in \Omega^k(\R^n)$ is a calibration.  If $U$ is bounded, then $U$ is $R_\phi$-hyperbolic.
\end{thm}

\indent We now consider calibrated analogues of Kiernan's theorem.  In a classic case of turning a theorem into a definition, a complex manifold $X$ is \emph{taut} if  $\mathrm{Hol}(B^2, X)$ is a normal family.  This means that every sequence of holomorphic maps $f_j \colon B^2 \to X$ either has a convergent subsequence (in the compact-open topology) or is compactly divergent.    Because taut manifolds satisfy a version of Montel's theorem by definition, they also enjoy its many consequences. In this way, taut manifolds can be viewed as a subclass of Kobayashi hyperbolic manifolds with a particularly rich function theory.  See~\cite{Abate} for this perspective. \\
\indent By analogy, we shall say that a calibrated manifold $(X, g_X, \phi)$ is \emph{$\phi$-taut} if $\mathrm{SmIm}(B^k, X)$ is a normal family.  When $(X, g_X, \omega)$ is a K\"{a}hler manifold, $\omega$-tautness is the usual notion of tautness for complex manifolds.  In $\S$\ref{sec:EquivHyp}, we prove:

\begin{thm}[Kiernan for calibrated manifolds] \label{thm:Kiernan-main}
Let $(X, g, \phi)$ be a calibrated manifold.
\begin{enumerate}[(a)]\setlength\itemsep{-0.2mm}
\item If $X$ is $\phi$-saturated, $K_{\phi}$-hyperbolic and $(X, d_{\phi})$ is complete, then $X$ is $\phi$-taut.
\item If $X$ is $\phi$-taut, then  $\mathrm{SmIm}(B^k,X)$ is an even family, and hence $X$ is $R_\phi$-hyperbolic.
\end{enumerate}
\end{thm}

\indent Finally, we remark that there also exist purely Riemannian notions of hyperbolicity that use conformal harmonic maps from surfaces in lieu of holomorphic curves or Smith immersions.  This theory is due to Forstneri\v{c}--Kalaj~\cite{FK}, Drnovsek--Forstneri\v{c}~\cite{DF}, and Forstneri\v{c}~\cite{For} in flat euclidean space, and to Gaussier--Sukhov~\cite{GS2023, GS2025} in general Riemannian manifolds.  In particular, Drnovsek--Forstneric~\cite{DF} and Gaussier--Sukhov~\cite{GS2023} have proven Riemannian analogues of Royden's and Kiernan's theorems, which inspired the present work.

\begin{rmk}
Independently, Broder--Hegarty--Hudecek (personal communication) proved Theorem \ref{thm:Kiernan-main}(a) and the equivalence of $R_\phi$-hyperbolicity and $K_\phi$-hyperbolicity under different assumptions. For compact calibrated manifolds, Ikonen--Pim prove the equivalence of $\phi$-, $R_\phi$-, and $K_\phi$-hyperbolicity in \cite[Theorem~1.3]{IkonenPim2026}. Their Theorem~1.2 gives pre-compactness of $\mathrm{SmIm}(B^k,X)$ when $X$ is compact and $\phi$-hyperbolic.
\end{rmk}

\subsection{Organization, notation, and conventions}

\indent \indent This work is organized as follows.  In $\S$\ref{sub:Smith-Prelim} and $\S$\ref{sub:Hyp-Prelim}, we review the basic aspects of Smith immersions and hyperbolicity notions in calibrated geometry.  In $\S$\ref{sub:Products}, we study the hyperbolicity of product manifolds. Section \ref{sec:Schwarz} is devoted to the proofs of Theorems \ref{thm:SchwarzMain} and \ref{thm:BddDomMain}. \\
\indent Section \ref{sub:SpacesCts} contains preliminary remarks on equicontinuity and normal families of continuous maps.  Then, in $\S$\ref{sub:EquivHyp}, we prove Theorem \ref{thm:RoydenMain}.  As an application, in $\S$\ref{sub:DiscreteQuot}, we study the hyperbolicity of discrete quotients.  Finally, $\S$\ref{sec:Tautness} contains the proof of Theorem \ref{thm:Kiernan-main}. \\

\noindent \textbf{Notation and conventions:}
\begin{itemize}
\item By ``manifold," we always mean a connected smooth manifold, unless indicated otherwise.  For us, smooth manifolds are assumed to be second countable and Hausdorff.
\item For a Riemannian manifold $(X, g)$, we let $\mathrm{dist}_g \colon X \times X \to [0,\infty)$ denote the standard Riemannian distance function.  When $X$ is oriented, we let $\vol_X \in \Omega^n(X)$ denote the volume form.
\item For a linear operator $A \colon V \to W$ between finite-dimensional inner product spaces, we denote its Hilbert-Schmidt and operator norms, respectively, by
\begin{align*}
|A| & = |A|_{\mathrm{HS}} = \sqrt{ \sum_{i=1}^n |Ae_i|^2 }, & \Vert A \Vert & = \Vert A \Vert_{\mathrm{op}} = \sup_{|v| = 1} \frac{|Av|}{|v|},
\end{align*}
where $\{e_1, \ldots, e_n\}$ is an orthonormal basis of $V$.  Note that when $A$ is conformal, we have that $\Vert A \Vert = \frac{1}{\sqrt{\dim(V)}}|A|$.
\item We let $B_r = \{x \in \R^k \colon |x| < r\}$ be the ball of radius $r$. For $r < R$, the notation $B_r \prec \varphi \prec B_R$ means that $\varphi$ is identically $1$ on $B_{r}$ and has support contained in $B_{R}$.
\item We reserve the notation $B^k = \{x \in \R^k \colon |x| < 1\}$ for the unit ball equipped with the Poincar\'{e} metric $g_1$ of constant curvature $-4$, and use $\mathrm{dist}_1$ to denote the induced Riemannian distance.  Explicitly, these are given by
\begin{align*}
g_1 & = \frac{1}{(1 - |x|^2)^2}\,g_0, & \mathrm{dist}_1(x,y) & = \mathrm{arcsinh}\sqrt{ \frac{|x-y|^2}{(1 - |x|^2)(1 - |y|^2)} },
\end{align*}
where $g_0$ is the flat metric. (An exception occurs in $\S$\ref{sec:Schwarz}, where $B^k$ is equipped with the flat metric.)
\end{itemize}

\noindent \textbf{Acknowledgments.} The authors are indebted to Da Rong Cheng for numerous valuable insights.  Portions of this work were conducted as part of the ``Research in Teams'' program at the Banff International Research Station (BIRS) in October 2025. The authors thank BIRS for their hospitality. The first author acknowledges the support of the International Scientists Project BJNSF \#25ISA023. The second author acknowledges the support of the Natural Sciences and Engineering Research Council of Canada (NSERC), grant number RGPIN-2025-03951. The third author thanks Kyle Broder and Toni Ikonen for helpful conversations. The authors also thank Pekka Pankka and the anonymous referee for suggesting improvements to earlier versions of this article. \\

\noindent \textbf{AI Disclosure.} Artificial intelligence (AI) was used to discuss possible improvements to the paper. The AI identified the $\phi$-saturation condition, which is weaker than the hypothesis used in an earlier version and suffices for the results proved here.

\section{Preliminaries}

\indent \indent In $\S$\ref{sub:Smith-Prelim} and $\S$\ref{sub:Hyp-Prelim}, we provide a rapid introduction to Smith immersions and hyperbolicity notions in calibrated geometry.  Then, in $\S$\ref{sub:Products}, we investigate the hyperbolicity of product manifolds.  The results in $\S$\ref{sub:Products} are new, and are independent of the rest of the paper.

\subsection{Smith immersions} \label{sub:Smith-Prelim}

\indent \indent The setting of this work is that of calibrated manifolds.

\begin{defn}[Harvey--Lawson~\cite{HL}] ${}$
\begin{itemize}
\item A \emph{calibrated manifold} $(X, g_X, \phi)$ is a Riemannian $n$-manifold $(X, g_X)$ equipped with a calibration $\phi \in \Omega^k(X)$.  By ``calibration," we mean a differential form $\phi$ that is closed ($d\phi = 0$) and has comass one.
\item Let $(X, g_X, \phi)$ be a calibrated manifold with $\phi \in \Omega^k(X)$.  A \emph{calibrated submanifold} is an oriented $k$-dimensional submanifold $\Sigma^k \subset X^n$ that satisfies $\phi|_\Sigma = \vol_\Sigma$.
\end{itemize}
\end{defn}

\begin{thm}[Fundamental theorem of calibrations~\cite{HL}]
Let $(X, g_X, \phi)$ be a calibrated manifold.  If $\Sigma \subset X$ is $\phi$-calibrated, then $\Sigma$ is a minimal submanifold.  Moreover, if $\Sigma$ is compact and without boundary, then $\Sigma$ is volume-minimizing in its homology class.
\end{thm}

\indent We shall take a conformal mapping approach to the study of calibrated submanifolds.  To explain this point of view, let us recall that a $C^1$ map $f \colon (\Sigma^k, g_\Sigma) \to (X^n, g_X)$ between Riemannian manifolds is called \emph{weakly conformal} if there exists a function $\lambda \colon \Sigma \to [0,\infty)$, called the \emph{conformal factor}, such that $f^*g_X = \lambda^2 g_\Sigma$.  In this case, the conformal factor at $p \in \Sigma$ is given by
\begin{equation*}
\lambda(p) = \frac{1}{\sqrt{k}}|df_p| = \Vert df_p \Vert.
\end{equation*}
If $f$ is weakly conformal, then at each point $p \in \Sigma$, we have either $\mathrm{rank}(df_p) = k$ or $\mathrm{rank}(df_p) = 0$.  Therefore, if $k > n$, then every weakly conformal map is locally constant.  As such, to avoid trivial edge cases, we adopt the following:

\begin{conv} Throughout this work, we use the index range $2 \leq k \leq n$, except where stated otherwise.
\end{conv}

\begin{defn}
Let $(\Sigma^k, g_\Sigma)$ be an oriented Riemannian $k$-manifold, and let $(X^n, g_X, \phi)$ be a calibrated $n$-manifold, where $\phi \in \Omega^k(X)$.  Note that $k = \deg(\phi) = \dim(\Sigma)$.
\begin{itemize}
\item A \emph{Smith immersion} is a $C^1$ map $f \colon \Sigma^k \to X^n$ for which there exists a function $\lambda \colon \Sigma \to [0,\infty)$ satisfying
\begin{align*}
f^*g_X & = \lambda^2 g_\Sigma, \\
f^*\phi & = \lambda^k \vol_\Sigma.
\end{align*}
Equivalently~\cite{IK}, a $C^1$ map $f \colon \Sigma \to X$ is a Smith immersion if and only if it satisfies the single equation
$$f^*\phi = \left( \frac{1}{\sqrt{k}}|df| \right)^k \,\vol_\Sigma.$$
\end{itemize}
Note that the nomenclature can be misleading, because a Smith immersion is only actually an immersion on the open subset of points in $\Sigma$ where $\lambda > 0$. A Smith immersion is also called a \emph{conformal $\phi$-curve} by some authors~\cite{Pankka, IP}, which is perhaps a better name going forward.
\end{defn}
Smith immersions of full rank are exactly the conformal parametrizations of $\phi$-calibrated submanifolds. That is:

\begin{prop}[\cite{CKM}]
If $f \colon (\Sigma^k, g_\Sigma) \to (X, g_X, \phi)$ is a Smith immersion of full rank $k$, then its image $f(\Sigma)$ is an immersed $\phi$-calibrated submanifold of $X$.  Conversely, if $\iota \colon \Sigma \to (X, g_X, \phi)$ is an immersed $\phi$-calibrated submanifold, then equipping $\Sigma$ with its induced metric $g_\Sigma = \iota^*g_X$ and volume form $\vol_\Sigma = \iota^*\phi$ makes $\iota$ into a Smith immersion of full rank.
\end{prop}

\indent It is a remarkable fact that every Smith immersion is a $k$-harmonic map.  Let us briefly recall this concept.

\begin{defn}
Let $f \colon (\Sigma^k, g_\Sigma) \to (X^n, g_X)$ be a $C^1$ map between Riemannian manifolds.
\begin{itemize}
\item The \emph{$k$-tension of $f$}, denoted $\tau_k(f)$, is the section of $f^*TX$ defined by
$$\tau_k(f) = \mathrm{div}(|df|^{k-2}df).$$
We say $f$ is \emph{$k$-harmonic} if $\tau_k(f) = 0$ in the weak sense.
\item If $\Sigma$ is compact, the \emph{$k$-energy of $f$} is defined by
$$E_k(f) = \frac{1}{ (\sqrt{k})^k}\int_\Sigma |df|^k\,\vol_\Sigma.$$
It is well-known that the critical points of the $k$-energy functional are precisely the $k$-harmonic maps.
\end{itemize}
\end{defn}

\indent It turns out that every $C^1$ map $f \colon (\Sigma^k, g_\Sigma) \to (X^n, g_X, \phi)$, where $\phi \in \Omega^k(X)$, satisfies the following energy inequality~\cite{CKM, IP}:
\begin{equation} \label{eq:EnergyInequality}
\left(\frac{1}{\sqrt{k}}|df|\right)^k \vol_\Sigma \geq f^*\phi.
\end{equation}
Moreover, equality holds if and only if $f$ is a Smith immersion.  These considerations lead to Smith's theorem, which is a conformal mapping analogue of the fundamental theorem of calibrations.

\begin{thm}[Smith's theorem~\cite{Smith}]
If $f \colon (\Sigma^k, g_\Sigma) \to (X^n, g_X, \phi)$ is a Smith immersion, then $f$ is a $k$-harmonic map.  Moreover, if $\Sigma$ is compact and without boundary, then $f$ is $k$-energy minimizing in its homology class.
\end{thm}

We will use the following mean value inequality of Duzaar--Fuchs on several occasions. We state it with the target hypotheses for the reader's convenience.

\begin{thm}[Mean value inequality for $p$-harmonic maps {\cite{DuzaarFuchs1990}}]
\label{thm:duzaar-fuchs}
Let $k\geq2$ and $p\geq2$. Let $M\subset\R^N$ be a closed, complete
$C^3$ submanifold, equipped with the induced metric. Let $A$ denote its
second fundamental form and $\Pi_q^\perp$ the orthogonal projection
onto $(T_qM)^\perp$, where $q \in M$. Suppose there are constants $\kappa,\kappa'\geq0$
such that
\[
|\Pi_q^\perp(q-q')|\leq\frac{\kappa}{2}|q-q'|^2
\quad\text{for all }q,q'\in M,
\qquad
\|\nabla A\|\leq\kappa'.
\]
Then there exist constants $\varepsilon_1>0$ and $C_0>0$, depending only on
$k$, $p$, $\kappa$, and $\kappa'$, such that for every $r>0$, every
$C^1$ $p$-harmonic map $u\colon B_r\subset\R^k\to M$ satisfying
\[
r^{p-k} \int_{B_r}|du|^p\,dx \leq\varepsilon_1,
\]
we have
\[
\sup_{B_{\frac{r}{2}}}|du|^p
\leq C_0r^{-k} \int_{B_r}|du|^p\,dx.
\]
\end{thm}

\subsection{Hyperbolicity in calibrated geometry} \label{sub:Hyp-Prelim}

\indent \indent We now discuss the notions of $R_\phi$-, $K_\phi$-, and $\phi$-hyperbolicity.

\begin{defn}[\cite{BIM}]
Let $(X, g_X, \phi)$ be a calibrated $n$-manifold, $\phi \in \Omega^k(X)$.  Equip the $k$-ball $B^k$ with the Poincar\'{e} metric $g_1$, and let $\mathrm{SmIm}(B^k, X) \subset C^1(B^k, X)$ denote the family of Smith immersions $f \colon (B^k, g_1) \to (X, g_X, \phi)$.
\begin{itemize}
\item The \emph{KR $\phi$-metric} $K_X \colon TX \to [0,\infty]$ is
$$K_X(v_p) = \inf\!\left\{ a > 0 \colon \exists f \in \mathrm{SmIm}(B^k, X) \text{ s.t. } f(0) = p,\, df_0(e_1) = \frac{1}{a}v\right\}\!.$$
Note that if there does not exist a Smith immersion $f \colon B^k \to X$ having $f(0) = p$ and $df_0(e_1) = bv_p$ for some $b > 0$, then $K_X(v_p) = \infty$.
\item We say that $X$ is \emph{$R_\phi$-hyperbolic} if for every $p \in X$, there exists a neighborhood $U \subset X$ of $p$, and a constant $c > 0$, such that $K_X(v) \geq c|v|$ for all $v \in TU$.
\item We say that $X$ is \emph{$\phi$-saturated} if $K_X \colon TX \to [0,\infty]$ is upper semicontinuous at every point of the zero section of $TX$. That is, if
\[
\limsup_{\substack{q\to p,\;v_q\in T_qX\\ |v_q|_{g_X}\to0}}K_X(v_q)=0
\qquad\text{for every }p\in X.
\]
\end{itemize}
\end{defn}

\indent Intuitively, $K_X(v_p)$ is the reciprocal of the largest possible conformal factor of a Smith $k$-disk having $v \in T_pX$ as a tangent vector.  In \cite{BIM}, it is shown that $K_X$ satisfies a ``decreasing property" analogous to that enjoyed by the Kobayashi-Royden pseudo-metric.  In general, $K_X$ is difficult to calculate explicitly, even in the K\"{a}hler case.

\begin{example}[K\"{a}hler manifolds]
Let $(X, g_X, \omega)$ be a K\"{a}hler manifold.  In this context, the KR $\omega$-metric is simply the Kobayashi-Royden pseudo-metric, and $R_\omega$-hyperbolicity is exactly Royden's criterion for Kobayashi hyperbolicity~\cite{Royden}.  Moreover, $K_X$ is upper semicontinuous on $TX$~\cite{Royden}, so every K\"{a}hler manifold is $\omega$-saturated.
\end{example}

The next corollary gives a large class of examples of $\phi$-saturated manifolds.
\begin{cor}\label{prop:flat-saturation}
Let $\phi\in\Lambda^k(\R^n)^*$ be a constant-coefficient elliptic calibration. That is, every vector in $\R^n$ lies in a $\phi$-calibrated $k$-plane. Let $U\subset\R^n$ be open, equipped with its Euclidean metric and calibration $\phi|_U$. Then $U$ is $\phi$-saturated.
\end{cor}
\begin{proof}
Fix $p\in U$ and choose $R>0$ such that $B_R(p)\subset U$. All norms below are taken with respect to $g_0$. By the decreasing property and the estimate for Euclidean balls in~\cite{BIM}, for $q\in B_{\frac{R}{2}}(p)$ and $v_q\in T_qU$ we have
\[
0\leq K_U(v_q)\leq K_{B_R(p)}(v_q)\leq\frac{|v_q|}{R-|q-p|}\leq\frac{2}{R}|v_q|.
\]
Thus $K_U(v_q)\to0$ as $q\to p$ and $|v_q|\to0$, with $v_q\in T_qU$, giving the required conclusion.
\end{proof}

Moreover, under an additional assumption on the calibration, $K_U$ is upper semicontinuous on all of $TU$.

\begin{prop}\label{prop:flat-usc}
Let $\phi\in\Lambda^k(\R^n)^*$ be a constant-coefficient calibration. Suppose that the stabilizer
\[
G_\phi:=\{A\in\SO(n):A^*\phi=\phi\}
\]
acts transitively on $S^{n-1}$. Then, for every open subset $U\subset\R^n$, equipped with $g_0$ and $\phi|_U$, the function $K_U\colon TU\longrightarrow[0,\infty)$ is upper semicontinuous.
\end{prop}
\begin{proof}
Since $\phi$ has comass one, there is some $\phi$-calibrated $k$-plane $P$. Fix $u\in P$ with $|u|=1$. For every $w\in S^{n-1}$, transitivity gives $A\in G_\phi$ with $Au=w$. The plane $A(P)$ is calibrated and contains $w$. Hence $\phi$ is elliptic. Corollary~\ref{prop:flat-saturation} gives upper semicontinuity of $K_U$ at every point of the zero section, and Lemma~\ref{lemma:lemmaK}(ii) gives finiteness of $K_U$.

Let $p_j\to p$ in $U$ and $v_j\in T_{p_j}U$ satisfy $|v_j-v|\to0$, where $v\in T_pU\setminus\{0\}$ and the tangent spaces are identified with $\R^n$. Let $f\in\mathrm{SmIm}(B^k,U)$ and $a>0$ satisfy
\[
f(0)=p,\qquad df_0(e_1)=\frac{v}{a}.
\]
Fix $\rho\in(0,1)$. The group $G_\phi$ is a closed subgroup of $\SO(n)$, hence a Lie group. By transitivity, the orbit map
\[
\pi_v\colon G_\phi\longrightarrow S^{n-1},\qquad \pi_v(A):=A \Big(\frac{v}{|v|} \Big),
\]
is a smooth submersion. By the Local Submersion Theorem, there is an open neighborhood $V\subset S^{n-1}$ of $\frac{v}{|v|}$ and a smooth local section $s\colon V\to G_\phi$ of $\pi_v$ satisfying
\[
\pi_v\circ s=\operatorname{id}_V,\qquad s\Big(\frac{v}{|v|}\Big)=I.
\]
For all sufficiently large $j$, we have $\frac{v_j}{|v_j|}\in V$. Set $A_j:=s\big(\frac{v_j}{|v_j|}\big)$. Then
\[
A_j\longrightarrow I,\qquad A_j \Big(\frac{v}{|v|}\Big) = \pi_v (A_j) =\frac{v_j}{|v_j|}.
\]
Set $t_j:=\frac{|v_j|}{|v|}$ and
\[
g_j(x):=p_j + t_jA_j\bigl(f(\rho x)-p\bigr),\qquad x\in B^k.
\]
Translations, positive dilations, and the maps $A_j$ preserve the Smith equations for $\phi$. Moreover, $f(\overline{B_\rho})$ is a compact subset of $U$ and
\[
\sup_{x\in B^k}|g_j(x)-f(\rho x)|
\leq |p_j-p|+\|t_j A_j - I\|
\max_{y\in\overline{B_\rho}}|f(y)-p|\longrightarrow0.
\]
Consequently $g_j(B^k)\subset U$ for all sufficiently large $j$. Since
\[
(dg_j)_0(e_1)=\frac{\rho}{a}v_j,
\qquad K_U(v_j)\leq\frac{a}{\rho},
\]
taking the infimum over admissible $a$ and then letting $\rho\uparrow1$ gives
\[
\limsup_{j\to\infty}K_U(v_j)\leq K_U(v),
\]
which is the required upper semicontinuity.
\end{proof}

\begin{example}\label{ex:flat-usc-calibrations}
Proposition~\ref{prop:flat-usc} applies to every open subset of the following Euclidean calibrated spaces, with the indicated transitive subgroups of $G_\phi$:
\begin{itemize}
\item $(\R^n,g_0,\vol_0)$, $n\geq2$, with subgroup $\SO(n)$.
\item $(\C^m,g_0,\frac{\omega_0^r}{r!})$, $1\leq r\leq m$, where $\omega_0$ is the standard K\"ahler form, with subgroup $\U(m)$.
\item $(\C^m,g_0,\operatorname{Re}(dz_1\wedge\cdots\wedge dz_m))$, $m\geq2$, with subgroup $\SU(m)$.
\item $(\R^7,g_0,\varphi_0)$ and $(\R^7,g_0,*\varphi_0)$, where $\varphi_0$ is the standard associative calibration and $*$ is the Euclidean Hodge star, with subgroup $G_2$.
\item $(\R^8,g_0,\Phi_0)$, where $\Phi_0$ is the standard Cayley calibration, with subgroup $\Spin(7)$ in its real eight-dimensional spin representation.
\item $(\mathbb H^m,g_0, \frac{6^r}{(2r+1)!} \Psi^r)$, $1\leq r\leq m$, where
\[
\Psi=\frac{1}{6}(\omega_I^2+\omega_J^2+\omega_K^2),
\qquad \omega_L(u,v)=g_0(Lu,v)\quad(L=I,J,K),
\]
and $I,J,K$ are the standard quaternionic complex structures, with subgroup $\Sp(m) \cdot \Sp(1)$.
\end{itemize}
\end{example}

\begin{prop}\label{prop:saturation-finite}
Let $(X,g_X,\phi)$ be calibrated. If $X$ is $\phi$-saturated, then $K_X$ is finite.
\end{prop}

\begin{proof}
By Lemma~\ref{lemma:lemmaK}(ii), every $p\in X$ has a neighborhood $U$ and a constant $C_U>0$ such that $K_X(v)\leq C_U|v|<\infty$ for all $v\in TU$.
\end{proof}

\begin{defn}[\cite{BIM}]\label{def:upper-distance}
For $h\colon[a,b]\to[0,\infty]$, its \emph{upper Lebesgue integral}~\cite{Fed} is defined to be
\[
\int_a^b{}^{*}h(t)\,dt
:=\inf\left\{\int_a^b g(t)\,dt:
g\colon[a,b]\to[0,\infty]\text{ is Lebesgue measurable and }g\geq h\right\}.
\]
Let $(X,g_X,\phi)$ be a calibrated manifold. For a piecewise $C^1$ curve $\gamma\colon[0,1]\to X$, set
\[
L_\phi(\gamma):=\int_0^1{}^{*}K_X(\gamma'(t))\,dt,
\]
with $\gamma'(t):=0$ at the endpoints of the $C^1$ pieces.
\begin{itemize}
\item The \emph{Kobayashi $\phi$-pseudo-distance} is $d_\phi \colon X \times X \to [0,\infty]$ given by
\[
d_\phi(p,q) = \inf_\gamma \int_0^1{}^{*} K_X(\gamma'(t))\,dt,
\]
where the infimum is taken over all piecewise $C^1$ paths $\gamma \colon [0,1] \to X$ having $\gamma(0) = p$ and $\gamma(1) = q$.
\item We say that $X$ is \emph{$K_\phi$-hyperbolic} if $d_\phi$ is non-degenerate (i.e., $p,q \in X$ with $p \neq q$ implies $d_\phi(p,q) > 0$).
\end{itemize}
\end{defn}

\begin{rmk} For a measurable $h$, the upper Lebesgue integral equals the usual Lebesgue integral. If $(X,g_X,\omega)$ is K\"ahler, then $K_X$ is upper semicontinuous on $TX$~\cite{Royden}, and $X$ is $\omega$-saturated. In this case, for every piecewise $C^1$ curve $\gamma$, the integrand $t \mapsto K_X(\gamma'(t))$ is bounded and measurable on $[0,1]$, and upper semicontinuous on the interior of each $C^1$ piece.
\end{rmk}

\begin{prop}\label{prop:distance-properties}
For every calibrated manifold $(X,g_X,\phi)$, $d_\phi$ is an extended pseudo-distance. Every Smith immersion $f\colon B^k\to X$ satisfies
\[
d_\phi(f(x),f(y))\leq\mathrm{dist}_1(x,y)
\qquad(x,y\in B^k).
\]
\end{prop}
\begin{proof}
By standard properties of upper integrals~\cite{Fremlin1,Fremlin2} and the homogeneity of $K_X$~\cite{BIM}, we have
\[
L_\phi(\gamma^{-1})=L_\phi(\gamma),\qquad
L_\phi(\gamma_1*\gamma_2)=L_\phi(\gamma_1)+L_\phi(\gamma_2),
\]
where $\gamma^{-1}(t):=\gamma(1-t)$ and $\gamma_1*\gamma_2$ is the concatenation of piecewise $C^1$ curves with $\gamma_1(1)=\gamma_2(0)$, each parametrized on half of $[0,1]$.
 Constant curves, reversal, and concatenation yield, respectively,
\[
d_\phi(p,p)=0,\qquad d_\phi(p,q)=d_\phi(q,p),\qquad
d_\phi(p,r)\leq d_\phi(p,q)+d_\phi(q,r).
\]
The infinitesimal decreasing property~\cite[Proposition 5.3]{BIM} and the KR metric of $(B^k,g_1,\vol_1)$ give
\[
K_X(df_x\xi)\leq |\xi|_{g_1}.
\]
For a $g_1$-minimizing geodesic $\sigma$ from $x$ to $y$, the curve $f\circ\sigma$ is $C^1$, and hence
\[
d_\phi(f(x),f(y))\leq L_\phi(f\circ\sigma)
\leq\int_0^1|\sigma'(t)|_{g_1}\,dt=\mathrm{dist}_1(x,y). \qedhere
\]
\end{proof}

\begin{thm} \label{thm:R-Implies-K} Suppose $(X, g_X, \phi)$ is a calibrated manifold.  If $X$ is $R_\phi$-hyperbolic, then $X$ is $K_\phi$-hyperbolic.
\end{thm}

\begin{proof}
Fix $p\ne q$. Choose a neighborhood $U\ni p$ and $c>0$ such that $K_X(v)\geq c|v|$ for $v\in TU$. Choose $r>0$ with $\overline{B_r(p)}\subset U$ and $q\notin\overline{B_r(p)}$. Every piecewise $C^1$ curve from $p$ to $q$ has an initial segment $\sigma$ from $p$ to $\partial B_r(p)$ contained in $\overline{B_r(p)}$. Monotonicity of the upper integral gives
\[
L_\phi(\gamma)\geq L_\phi(\sigma)\geq c\int|\sigma'(t)|\,dt\geq cr.
\]
Taking the infimum over $\gamma$ gives $d_\phi(p,q)\geq cr>0$.
\end{proof}

\begin{rmk}
    In~\cite{BIM}, both Proposition~\ref{prop:distance-properties} and Theorem~\ref{thm:R-Implies-K} were proven under the additional assumption that $K_X$ is upper semicontinuous everywhere.
\end{rmk}

\indent Finally, we briefly recall the calibrated analogue of Brody hyperbolicity.

\begin{defn}[\cite{BIM}]
Let $(X, g_X, \phi)$ be a calibrated manifold with $\phi \in \Omega^k(X)$.  We say that $X$ is \emph{$\phi$-hyperbolic} if every Smith immersion $f \colon \R^k \to X$ is constant, where here $\R^k$ carries the flat metric.
\end{defn}

\begin{rmk}
When $(X, g_X, \omega)$ is a K\"{a}hler manifold, $\omega$-hyperbolicity is exactly Brody hyperbolicity.  As is well-known, many classical theorems in complex analysis can be phrased in this language.  For example, the Little Picard Theorem asserts that every open subset $U \subset \C \setminus \{\text{two points}\}$ is Brody hyperbolic.  In particular, every bounded domain in $\C$ is Brody hyperbolic (Liouville's theorem).
\end{rmk}

In \cite{BIM}, it was shown that $R_\phi$-hyperbolicity implies $\phi$-hyperbolicity. In fact, more is true.

\begin{prop}\label{prop:K-Brody}
Let $(X, g_X, \phi)$ be a calibrated manifold. If $X$ is $K_\phi$-hyperbolic, then $X$ is $\phi$-hyperbolic.
\end{prop}

\begin{proof}
Let $f\colon\R^k\to X$ be a Smith immersion and fix $x,y\in\R^k$. For $R>\max\{|x|,|y|\}$, the map $f_R(z):=f(Rz)$ is a Smith immersion on $B^k$. Proposition~\ref{prop:distance-properties} gives
\[
d_\phi(f(x),f(y))\leq\mathrm{dist}_1 \Big(\frac{x}{R},\frac{y}{R}\Big)\longrightarrow0
\quad(R\to\infty).
\]
Nondegeneracy implies $f(x)=f(y)$.
\end{proof}

\subsection{Product spaces} \label{sub:Products}

\indent \indent We now turn to the hyperbolicity of product manifolds.  The material in this section is new.  Throughout, we let $(X^m, g_X, \alpha)$ and $(Y^n, g_Y, \beta)$ be calibrated manifolds equipped with calibrations of the same degree $k \geq 2$, i.e.:
$$k = \deg(\alpha) = \deg(\beta) \geq 2.$$
We always equip the Riemannian product $(X \times Y, g_X \oplus g_Y)$ with the $k$-form $\phi = \pi_X^*\alpha + \pi_Y^*\beta \in \Omega^k(X \times Y)$.  In fact, $\phi$ is a calibration:

\begin{prop}\label{prop:product}
If $k \geq 2$, then $\phi = \pi_X^* \alpha + \pi_Y^* \beta$ is a calibration on $(X \times Y, g_X \oplus g_Y)$.
\end{prop}
\begin{proof}
Since $\alpha$ and $\beta$ are closed, $d\phi = \pi_X^* d\alpha + \pi_Y^* d\beta = 0$.  It remains to show that $\phi$ has comass one.
\begin{itemize}
\item{($\operatorname{comass} \leq 1$)}
Let $\xi = (v_1, w_1) \wedge \cdots \wedge (v_k, w_k)$ be a unit simple $k$-vector at $(p, q) \in X \times Y$, where $v_i \in T_p X$ and $w_i \in T_q Y$. Write $G_V = (\langle v_i, v_j \rangle_{g_X})$ and $G_W = (\langle w_i, w_j \rangle_{g_Y})$ for the $k \times k$ Gram matrices. Both are positive semidefinite.  The product metric gives
\[
  |\xi|^2 = \det(G_V + G_W),\qquad
  |v_1 \wedge \cdots \wedge v_k|^2 = \det G_V,\qquad
  |w_1 \wedge \cdots \wedge w_k|^2 = \det G_W.
\]
Therefore,
\begin{align}
\phi(\xi) & = \alpha(v_1 \wedge \cdots \wedge v_k)
               + \beta(w_1 \wedge \cdots \wedge w_k) \notag \\ 
& \leq |v_1 \wedge \cdots \wedge v_k| + |w_1 \wedge \cdots \wedge w_k| \notag \\ 
& = \sqrt{\det G_V} + \sqrt{\det G_W}. \label{eq:bound1}
\end{align}

We now show
\begin{equation}\label{eq:bound2}
  \sqrt{\det G_V} + \sqrt{\det G_W}
  \leq \sqrt{\det(G_V + G_W)} = |\xi| = 1.
\end{equation}
By the Minkowski determinant inequality for positive semidefinite $k \times k$ matrices,
\[
  \bigl(\det(G_V + G_W)\bigr)^{\frac{1}{k}}
  \geq (\det G_V)^{\frac{1}{k}} + (\det G_W)^{\frac{1}{k}}.
\]
Set $a := (\det G_V)^{\frac{1}{k}}$ and $b := (\det G_W)^{\frac{1}{k}}$.  Then
\[
  \sqrt{\det G_V} + \sqrt{\det G_W} = a^{\frac{k}{2}}+b^{\frac{k}{2}} \leq (a+b)^{\frac{k}{2}} \leq \det(G_V + G_W)^{\frac{1}{2}},
\]
where for the first inequality we need $k \geq 2$ and the second is Minkowski.
Combining with~\eqref{eq:bound1} and~\eqref{eq:bound2} gives $\phi(\xi) \leq 1$.

\item{($\operatorname{comass}=1$ attained)}
Consider any point $(p,q)\in X \times Y$. Since $\alpha$ is a calibration, there exists an oriented unit $k$-plane in $T_p X$ with calibrating $k$-vector $\zeta_X = v_1 \wedge \cdots \wedge v_k$ satisfying $\alpha(\zeta_X) = 1$ and $|\zeta_X| = 1$. Setting $\xi = (v_1, 0) \wedge \cdots \wedge (v_k, 0)$, we have $|\xi| = |\zeta_X| = 1$ and $\phi(\xi) = \alpha(\zeta_X) + \beta(0) = 1$. Hence $\operatorname{comass}(\phi) \equiv 1$.
\end{itemize}
\end{proof}

\begin{rmk}[A Counterexample for $k = 1$]
For $k = 1$, the product form $\phi = \pi_X^* \alpha + \pi_Y^* \beta$ need not be a calibration, even when $\alpha$ and $\beta$ are calibrations. \\
\indent  For example, on $X = Y = \mathbb{R}$ with the standard metric, consider the calibrations $\alpha = dx$ and $\beta = dy$, where $x, y$ are the standard coordinates.  On $X \times Y = \mathbb{R}^2$, the product form is $\phi = dx + dy$.  For any unit tangent vector $(a, b) \in \mathbb{R}^2$, the supremum of $\phi(a,b) = a+b$ subject to $|(a,b)| = 1$ is $\sqrt{2}$, attained at $(a,b) = (\frac{1}{\sqrt{2}}, \frac{1}{\sqrt{2}})$.  Thus, $\operatorname{comass} (\phi) = \sqrt{2} > 1$, so $\phi$ is not a calibration.
\end{rmk}

\indent In degree $k \geq 3$, Smith immersions into product calibrated manifolds are very restricted, as we now show.

\begin{prop}[Smith immersions into product spaces] \label{prop:products} 
Let $F \colon \Sigma \rightarrow X \times Y$ be a Smith immersion, where $X \times Y$ is equipped with the product metric and calibration. Then the two projections $F_X := \pi_X \circ F$ and $F_Y := \pi_Y \circ F$ are both Smith immersions. Moreover, if $k \geq 3$, and $dF_p \neq 0$, then at least one of $F_X$ or $F_Y$ is constant in a neighborhood of $p$.
\end{prop}
\begin{proof}
Since $F$ is a Smith immersion, we have
\begin{equation} \label{eq:eq1}
F^*(\pi_X^{*} \alpha + \pi_Y^{*}\beta) = \frac{|dF|^k}{(\sqrt{k})^k} \mathrm{vol}_\Sigma.
\end{equation}
Also,
\begin{equation} \label{eq:eq2}
 |dF|^2 = |dF_X|^2 + |dF_Y|^2.
\end{equation}
Using the energy inequalities  (\ref{eq:EnergyInequality}) for $F_X$ and $F_Y$, we get
\begin{align}
 F^{*}_X \alpha &\leq \frac{|dF_X|^k}{(\sqrt{k})^k} \mathrm{vol}_\Sigma, & F^{*}_Y \beta &\leq \frac{|dF_Y|^k}{(\sqrt{k})^k} \mathrm{vol}_\Sigma. \label{eq:eq4}
\end{align}
Thus, we obtain the following chain of pointwise inequalities
\begin{align*}
|dF|^k \mathrm{vol}_\Sigma &=(\sqrt{k})^k F^*(\pi_X^{*} \alpha + \pi_Y^{*}\beta) & & \text{(by~\eqref{eq:eq1})} \\
&= (\sqrt{k})^k \left(F^{*}_X \alpha + F^{*}_Y \beta\right) \\
&\leq \left(|dF_X|^k +|dF_Y|^k\right) \mathrm{vol}_\Sigma & & \text{(by~\eqref{eq:eq4})} \\
&\leq \left(\sqrt{|dF_X|^2 +|dF_Y|^2}\right)^k \mathrm{vol}_\Sigma\\
&= |dF|^k \mathrm{vol}_\Sigma, & & \text{(by~\eqref{eq:eq2})}
\end{align*}
where in the second inequality above we used that $(a^k + b^k)^2 \leq (a^2 + b^2)^k$ for $a,b \geq 0$. (This equality is trivial when $k=2$. When $k > 2$, equality occurs if and only if $a=0$ or $b=0$.) For $k\geq2$, the chain of inequalities implies that $F_X$ and $F_Y$ are both Smith immersions.

For $k > 2$, we deduce that at every point of $\Sigma$, at least one of $dF_X$ or $dF_Y$ must vanish. Suppose $dF_p \neq 0$, so exactly one of $dF_X$ or $dF_Y$ vanishes at $p$. Without loss of generality, suppose $(dF_X)_p \neq 0$. Then $dF_X$ is nonzero in a neighborhood of $p$, so $dF_Y$ is zero on that neighborhood.
\end{proof}

To address the hyperbolicity of product manifolds carrying a product calibration of degree $k \geq 3$, we require the following lemma.

\begin{lem} \label{lem:KXY}
Suppose $k \geq 3$.  Let $v \in TX$ and $w \in TY$.
For the product $X \times Y$, the following formula holds:
\[
K_{X\times Y}(v,w)=
\begin{cases}
K_X(v), & w=0,\\
K_Y(w), & v=0,\ w\neq0,\\
\infty, & v\neq0,\ w\neq0.
\end{cases}
\]
\end{lem}
\begin{proof}
Fix $p\in X$, $q\in Y$, $v\in T_pX$, and $w\in T_qY$.
If both $v$ and $w$ are nonzero, a Smith immersion
$F\colon B^k\to X\times Y$ with $F(0)=(p,q)$ and
$dF_0(e_1)=\frac{1}{a}(v,w)$ for some $a>0$ would have both
$(dF_X)_0(e_1)=\frac{1}{a}v\neq0$ and $(dF_Y)_0(e_1)=\frac{1}{a}w\neq0$.
This contradicts Proposition~\ref{prop:products}, which says that at least
one of $(dF_X)_0$ and $(dF_Y)_0$ vanishes. Thus, $K_{X\times Y}(v,w)= \inf\varnothing = \infty$.

When $w=0$, we show that the set of positive numbers $a$ in the definitions of $K_{X\times Y}(v,0)$ and $K_X(v)$ are the same.
Indeed, any Smith immersion $F$ with $dF_0(e_1) = \frac{1}{a} (v,0)$
projects to a Smith immersion $F_X$ having $(dF_X)_0(e_1)= \frac{1}{a} v$, by
Proposition~\ref{prop:products}.
Conversely, if a Smith immersion $f\colon B^k\to X$ satisfies
$f(0)=p$ and $df_0(e_1)=\frac{1}{a} v$, define $F(x):=(f(x),q)$.
Writing $\lambda$ for the conformal factor of $f$, we have
\[
F^*(g_X\oplus g_Y)=f^*g_X=\lambda^2g_{1},
\qquad
F^*\phi=f^*\alpha=\lambda^k\vol_{B_1},
\]
since the $Y$-component is constant. Hence $F$ is a Smith immersion and has
$dF_0(e_1) = \frac{1}{a} (v,0)$. Taking infima over the identical sets of
such $a$'s gives $K_{X\times Y}(v,0)=K_X(v)$, also when these
sets are empty. Similarly we get $K_{X\times Y}(0,w)=K_Y(w)$.
\end{proof}

We arrive at the main result of this section.

\begin{prop}[Hyperbolicity of products]\label{prop:product-hyperbolicity}
Let $(X^m, g_X, \alpha)$ and $(Y^n, g_Y, \beta)$ be calibrated manifolds equipped with calibrations of the same degree $k \geq 3$. Equip $(X \times Y, g_X \oplus g_Y)$ with the product calibration $\phi = \pi_X^*\alpha + \pi_Y^*\beta \in \Omega^k(X \times Y)$.
\begin{enumerate}[(a)]
\item If $X$ is $\alpha$-hyperbolic and $Y$ is $\beta$-hyperbolic, then $X\times Y$ is $\phi$-hyperbolic.
\item If $X$ is $R_{\alpha}$-hyperbolic and $Y$ is $R_{\beta}$-hyperbolic, then $X\times Y$ is $R_{\phi}$-hyperbolic.
\item The product $X \times Y$ is \emph{never} $\phi$-saturated.
\item The product pseudo-distance satisfies
\begin{equation}\label{eq:product-distance}
d_\phi((p_1,q_1),(p_2,q_2))
=d_\alpha(p_1,p_2)+d_\beta(q_1,q_2).
\end{equation}
In particular, if $X$ is $K_\alpha$-hyperbolic and $Y$ is $K_\beta$-hyperbolic, then $X \times Y$ is $K_\phi$-hyperbolic.
\end{enumerate}
\end{prop}

\begin{proof} ${}$
\begin{enumerate}[(a)]
\item We prove the contrapositive. Suppose $X \times Y$ is not $\phi$-hyperbolic. Then there exists a non-constant Smith immersion $F \colon\mathbb{R}^k \rightarrow X\times Y$. By Proposition~\ref{prop:products}, both $F_X$ and $F_Y$ are Smith immersions, and they cannot both be constant.  Without loss of generality, suppose $F_X$ is non-constant.  Then $F_X \colon \R^k \to X$ is a non-constant Smith immersion, so $X$ is not $\alpha$-hyperbolic.
\item Fix $(p,q)\in X\times Y$. Since $X$ is $R_{\alpha}$-hyperbolic and $Y$ is $R_{\beta}$-hyperbolic, there exist neighborhoods $U_X \subset X$ of $p$ and $U_Y \subset Y$ of $q$ and $c_X, c_Y > 0$ such that $K_X(v) \geq c_X |v|$ for $v \in TU_X$ and $K_Y(w) \geq c_Y |w|$ for $w \in TU_Y$. It follows from Lemma \ref{lem:KXY} and these inequalities that
$$ K_{X\times Y}(v,w) \geq \max\{K_X(v),K_Y(w)\} \geq \min\{c_X,c_Y\} \max\{|v|,|w|\} \geq \frac{1}{\sqrt{2}}\min\{c_X,c_Y\} \left|(v,w)\right|$$
on $TU_X \oplus TU_Y$, so $X \times Y$ is $R_{\phi}$-hyperbolic.

\item Choose nonzero $v\in T_pX$ and $w\in T_qY$. Lemma~\ref{lem:KXY} gives
\[
K_{X\times Y}(tv,tw)=\infty\quad(t>0),\qquad
|(tv,tw)|_{g_X\oplus g_Y}\longrightarrow0\quad(t\downarrow0).
\]
Since $K_{X\times Y}(0_p,0_q)=0$, upper semicontinuity fails at $(0_p,0_q)$.
\item Since upper integrals are in general only subadditive, we use additivity over measurable partitions to obtain the reverse inequality. Let $\gamma=(\gamma_X,\gamma_Y)$ be a piecewise $C^1$ curve from $(p_1,q_1)$ to $(p_2,q_2)$. If $L_\phi(\gamma)=\infty$, then
$L_\phi(\gamma)\geq d_\alpha(p_1,p_2)+d_\beta(q_1,q_2)$.
Suppose $L_\phi(\gamma)<\infty$. On the interiors of its $C^1$ pieces, the set
\[
E:=\{t:\gamma_X'(t)\ne0,\ \gamma_Y'(t)\ne0\}
\]
is open. If $E$ were nonempty, it would contain an interval, on which $K_{X\times Y}(\gamma'(t))=\infty$ by Lemma~\ref{lem:KXY}, contradicting finite length. Thus, with derivatives set to zero at the endpoints of the $C^1$ pieces, the measurable sets
\[
A:=\{t:\gamma_X'(t)\ne0\},\qquad
B:=\{t:\gamma_Y'(t)\ne0\}
\]
are disjoint. The integrands $K_X(\gamma_X'(t))$ and $K_Y(\gamma_Y'(t))$ vanish outside $A$ and $B$, respectively. Additivity of the upper integral over the partition $A,B,[0,1]\setminus(A\cup B)$ gives

\[
L_\phi(\gamma)
=\int_0^1{}^{*}K_X(\gamma_X'(t))\,dt
+\int_0^1{}^{*}K_Y(\gamma_Y'(t))\,dt
\geq d_\alpha(p_1,p_2)+d_\beta(q_1,q_2).
\]

Taking the infimum over $\gamma$ proves the lower bound. If the sum on the right is infinite, this already proves equality. Otherwise, fix $\varepsilon>0$ and choose piecewise $C^1$ paths $\sigma_X$ on $X$ and $\sigma_Y$ on $Y$ joining the respective points, with
\[
\int_0^1{}^{*} K_X(\sigma_X'(t))\,dt<d_\alpha(p_1,p_2)+\varepsilon,
\qquad
\int_0^1{}^{*} K_Y(\sigma_Y'(t))\,dt<d_\beta(q_1,q_2)+\varepsilon.
\]
 Concatenate $(\sigma_X,q_1)$ with $(p_2,\sigma_Y)$ and reparametrize to have domain $[0,1]$, and denote this path by $\sigma$. Homogeneity of $K_X$ and $K_Y$ and Lemma~\ref{lem:KXY} give
 \[
 \int_0^1{}^{*} K_{X \times Y}(\sigma'(t))\,dt= \int_0^1{}^{*} K_X(\sigma_X'(t))\,dt + \int_0^1{}^{*} K_Y(\sigma_Y'(t))\,dt.
 \]
 Hence
\[
d_\phi((p_1,q_1),(p_2,q_2)) \leq \int_0^1{}^{*} K_{X \times Y}(\sigma'(t))\,dt
\leq d_\alpha(p_1,p_2)+d_\beta(q_1,q_2)+2\varepsilon.
\]
Letting $\varepsilon\rightarrow 0$ proves equality. Non-degeneracy of $d_{\phi}$ follows immediately when both $d_{\alpha}$ and $d_{\beta}$ are non-degenerate. \qedhere
\end{enumerate}

\end{proof}

\section{Applications of mean-value inequalities} \label{sec:Schwarz}

\indent \indent In \cite{BIM}, the authors proved a Schwarz lemma for Smith immersions $f \colon \Sigma \to X$ assuming certain curvature bounds on the domain and target.  However, for our applications in this work, we need a gradient bound for Smith immersions $f \colon B^k \to X$ that swaps curvature assumptions on $X$ for an equicontinuity-type condition.  Such an estimate is provided by Theorem \ref{thm:SchwarzMain}, which we prove in $\S$\ref{sub:SecondSchwarz} (see Theorem \ref{lem:lemJKS}).  Then, in $\S$\ref{sub:BoundedDomains}, we prove Theorem \ref{thm:BddDomMain} (see Corollary  \ref{cor:bounded-domains-R-hyperbolic}). Our proofs of these two results are similar in spirit; both hinge on a particular mean-value inequality.

\begin{conv} Throughout this section, $B^k$ is equipped with the \emph{flat metric}.
\end{conv}

\subsection{A second Schwarz lemma} \label{sub:SecondSchwarz}

\indent \indent Cheng--Karigiannis--Madnick~\cite[Theorem 4.9]{CKM} prove a small-energy mean value inequality for associative Smith immersions by adapting the argument of Duzaar--Fuchs (Theorem~\ref{thm:duzaar-fuchs}); see~\cite[Appendix B]{CKM}. Their result starts with $W^{1,3}$ maps, allows a variable domain metric subject to the stated bounds, and does not require the target calibration $3$-form to be closed. They first establish $C^{1,\beta}$ regularity and then adapt the gradient estimate to their equation.

In our setting the domain metric is flat, Smith immersions are $C^1$, and the calibration is closed, so the maps are $k$-harmonic. The estimate we require follows directly from Theorem~\ref{thm:duzaar-fuchs} with $p=k$. We state it here with the image restricted to a fixed relatively compact target neighborhood. The proof below explains how to apply that theorem with uniform target bounds.

\begin{thm}[Mean Value Inequality for Smith immersions] \label{thrm:mvt}
Let $(X,g_X,\phi)$ be a calibrated manifold and fix a relatively compact
open subset $U\Subset X$. There exist $C>0$ and $\varepsilon>0$ such that,
for every $r>0$, every Smith immersion $u\colon B_r(0)\subset\R^k\to U$ with
\[
\int_{B_r}|du|^k<\varepsilon
\]
satisfies
\[
\sup_{x\in B_{\frac{r}{2}}}|du(x)|^k
\leq \frac{C}{r^k}\int_{B_r}|du|^k.
\]
The constants depend only on $k$ and the fixed target neighborhood $U$.
\end{thm}
\begin{proof}
Smith immersions are $C^1$ and $k$-harmonic, so the estimate follows from Theorem~\ref{thm:duzaar-fuchs} with energy exponent and domain
dimension both equal to $k$. To fix the target constants, choose a compact
smooth domain $V$ in $X$ containing $\overline U$ in its interior. Modify the metric on $V$ only near its boundary and glue two copies of $V$ to each other along their boundaries to obtain a compact Riemannian manifold whose metric agrees with $g_X$ on a neighborhood of $\overline U$. After an isometric embedding of this compact manifold into Euclidean space, the geometric bounds required by Theorem~\ref{thm:duzaar-fuchs} are uniform. Maps with image in $U$ remain $k$-harmonic for the new metric. Using $p=k$, the cited theorem gives the desired conclusion.
\end{proof}

\begin{thm}[Schwarz Lemma for Smith immersions] \label{lem:lemJKS}
Let $(X,g_X,\phi)$ be calibrated. For every $p\in X$, there exist
$R_0>0$ and $C>0$ with the following property. Let $0<R\leq R_0$.
If there exist radii $\delta\in(0,1)$ and $\delta'\in(0,R)$ such that
all Smith immersions $f\colon B^k\to X$ with $f(0)\in B_{\delta'}(p)$
satisfy $f(B_\delta(0))\subset B_R(p)$, then every such $f$ satisfies
\[
\Vert df_0\Vert\leq C\frac{R}{\delta}.
\]

\end{thm}
\begin{proof}
Fix a point $p\in X$ and a relatively compact neighborhood $U\Subset X$ of $p$.
Choose $r_*>0$ such that $\overline{B_{r_*}(p)}\subset U$, the ball is a
normal neighborhood of $p$, and the function
$\rho(q):=\frac{1}{2}\mathrm{dist}_{g_X}(p,q)^2$ is smooth there with
\begin{equation}\label{eq:local-hessian}
\mathrm{Hess}_{g_X}\rho\geq\frac{1}{2} g_X,
\qquad |\nabla\rho(q)|=\mathrm{dist}_{g_X}(p,q)
\quad\text{for $q\in B_{r_*}(p)$}.
\end{equation}
Such a choice is possible because
$(\mathrm{Hess}_{g_X}\rho)(p)=g_X(p)$.
Let $C_U$, $\varepsilon_U>0$ be the constants in
Theorem~\ref{thrm:mvt} for this fixed $U$.
Fix a smooth cutoff $B_{\frac{1}{2}}\prec\eta\prec B_1=B^k$ with
$0\leq\eta\leq1$. Recall that this means that $\eta\equiv1$ on
$B_{\frac{1}{2}}$ and $\eta$ has compact support in $B_1$. Set
\begin{equation}\label{eq:local-cutoff-constant}
A_k:=(2k)^k\int_{B^k}|d\eta|^k.
\end{equation}
Choose $R_0$ so that
\begin{equation}\label{eq:local-radius}
0<R_0<r_*,\qquad A_kR_0^k<\varepsilon_U.
\end{equation}

Now let $0<R\leq R_0$ and let $f$ satisfy the hypotheses.
By conformal invariance of the Smith equation, the rescaled map
$u(x):=f(\delta x)$ is a Smith immersion from the flat $B^k$,
with $u(B^k)\subset B_R(p)\subset U$; in particular it is $k$-harmonic.
Use the compactly supported $C^1$ section of $u^*TX$ given by
\begin{equation}\label{eq:local-test-section}
V:=\eta^k(\nabla\rho)\circ u
\end{equation}
as a test section in the weak equation (by approximation, $C^1$ test
sections are admissible). Writing $e_i$ for the standard orthonormal
frame on $B^k$ and $\nabla^u$ for the pullback connection, we obtain
\begin{align}
0 & = \int_{B^k}|du|^{k-2}\sum_{i=1}^k
  \langle du(e_i),\nabla^u_{e_i}V\rangle \notag \\
 & {}= \int_{B^k}\eta^k|du|^{k-2}\sum_{i=1}^k
  \langle du(e_i),\nabla^X_{du(e_i)}\nabla\rho\rangle \notag \\
 &\qquad {}+ k\int_{B^k}\eta^{k-1}|du|^{k-2}\sum_{i=1}^k
  (e_i\eta)\langle du(e_i),(\nabla\rho)\circ u\rangle \notag \\
 & = \int_{B^k}\eta^k|du|^{k-2}\sum_{i=1}^k
  \mathrm{Hess}\,\rho(du(e_i),du(e_i)) \notag \\
 &\qquad {}+ k \int_{B^k}\eta^{k-1}|du|^{k-2}
  \langle d\eta,d(\rho\circ u)\rangle. \label{eq:local-weak-equation}
\end{align}

By the Hessian bound of~\eqref{eq:local-hessian}, we have
\[
\sum_{i=1}^k
\mathrm{Hess}\,\rho(du(e_i),du(e_i))
\geq \frac{1}{2}\sum_{i=1}^k|du(e_i)|^2
=\frac{1}{2}|du|^2.
\]
Moreover, since $u(B^k)\subset B_R(p)$, the chain rule gives
\[
|d(\rho\circ u)|
\leq |\nabla\rho\circ u|\,|du|
\leq R|du|.
\]
Thus, by Cauchy--Schwarz,
\[
|\langle d\eta,d(\rho\circ u)\rangle|
\leq R|d\eta|\,|du|.
\]
Rearranging~(3.5) and applying these estimates, we obtain
\begin{equation}\label{eq:local-energy-inequality}
\frac{1}{2}\int_{B^k}\eta^k|du|^k
\leq kR\int_{B^k}\eta^{k-1}|du|^{k-1}|d\eta|.
\end{equation}
Set $E_\eta:=\int_{B^k}\eta^k|du|^k$, which is finite since
 $u$ is $C^1$ and $\eta$ has compact support. Apply H\"older's inequality to the right hand side of~\eqref{eq:local-energy-inequality}, with factors $\eta^{k-1}|du|^{k-1}$ and $|d\eta|$, and conjugate exponents $\frac{k}{k-1}$ and $k$, to get
\begin{equation}\label{eq:local-holder}
\frac{1}{2} E_\eta
\leq kR E_\eta^{\frac{k-1}{k}}
\left(\int_{B^k}|d\eta|^k\right)^{\frac{1}{k}}.
\end{equation}
If $E_\eta=0$, the desired bound is immediate. Otherwise, divide by
$E_\eta^{\frac{k-1}{k}}$. This gives

$$ E_{\eta}^{\frac{1}{k}} \leq 2 k R \left(\int_{B^k}|d\eta|^k\right)^{\frac{1}{k}}. $$

Raising both sides to the $k$th power gives
\begin{equation}\label{eq:local-small-energy}
\int_{B_{\frac{1}{2}}}|du|^k\leq E_\eta
\leq (2k)^k R^k \left(\int_{B^k}|d\eta|^k\right) = A_kR^k\leq A_kR_0^k<\varepsilon_U.
\end{equation}

The image of $u|_{B_{\frac{1}{2}}}$ lies in the same fixed $U$.
Theorem~\ref{thrm:mvt}, applied with $r=\frac{1}{2}$, therefore gives
\begin{equation}\label{eq:local-gradient}
|du(0)|^k\leq2^kC_U\int_{B_{\frac{1}{2}}}|du|^k
\leq2^kC_UA_kR^k.
\end{equation}
Finally, $du(0)=\delta\,df(0)$, and weak conformality gives
\begin{equation}\label{eq:local-schwarz}
\Vert df_0\Vert=\frac{|df(0)|}{\sqrt{k}}
\leq \frac{(2^kC_UA_k)^{\frac{1}{k}}}{\sqrt{k}}\frac{R}{\delta},
\end{equation}
which finishes the proof.
\end{proof}

\subsection{Bounded domains in $\R^n$ are $R_\phi$-hyperbolic} \label{sub:BoundedDomains}

\indent \indent Our next goal is to prove that bounded domains in euclidean space are $R_\phi$-hyperbolic for any calibration $\phi \in \Omega^k(\R^n)$.  In our approach, the key estimate (Proposition \ref{prop:bounded-k-harmonic-estimate}) follows from a mean value inequality for $k$-harmonic maps.

\begin{prop}[Mean Value Inequality for $k$-harmonic maps into $\mathbb{R}^n$] \label{cor:mvt}
There exists $C_k>0$, depending only on $k$, such that for every $r>0$
and every $C^1$ $k$-harmonic map
$f\colon B_r(0)\subset \R^k \to \R^n$, we have
\[
\sup_{x\in B_{\frac{r}{2}}}|df(x)|^k
\leq \frac{C_k}{r^k}\int_{B_r}|df|^k.
\]
\end{prop}

\begin{proof}
This estimate follows from Theorem~\ref{thm:duzaar-fuchs}. Indeed, if $\int_{B_r}|df|^k$ is finite, choose $\lambda>0$ sufficiently small such that $\int_{B_r}|d(\lambda f)|^k<\varepsilon_1$, where $\varepsilon_1$ is the constant in Theorem~\ref{thm:duzaar-fuchs}. The map $\lambda f$ remains $k$-harmonic into $\R^n$. Since $\R^n$ is totally geodesic in $\R^n$, the constants $\kappa$ and $\kappa'$ of Theorem~\ref{thm:duzaar-fuchs} are both zero. Applying that theorem with both the energy exponent and the domain dimension equal to $k$ gives
\[
\lambda^k\sup_{B_{\frac{r}{2}}}|df|^k
\leq C_k r^{-k}\lambda^k\int_{B_r}|df|^k.
\]
Canceling $\lambda^k$ gives the result.
\end{proof}

\indent As a consequence of the mean value inequality, we obtain the following key estimate.

\begin{prop} \label{prop:bounded-k-harmonic-estimate}
Let $f \colon B^k \rightarrow \R^n$ be a $C^1$ $k$-harmonic  map. If the image of $f$ is contained in the ball $B_R$ for some $R > 0$, then
\begin{equation} \label{claim:claim1}
 \|df_0\| \leq C_k R,
\end{equation}
for some $C_k > 0$ depending only on $k$.
\end{prop}
\begin{proof}
Take a smooth cut-off function $B_{\frac{1}{2}} \prec \varphi \prec B_1 = B^k$. Since $f$ is $k$-harmonic we have
\begin{align*}
 0 &= \int_{B_1} \langle |df|^{k-2} df , d(\varphi^k f) \rangle \\
 &= \int_{B_1} \big( \varphi^k |df|^k + k \varphi^{k-1} |df|^{k-2} \langle df, d\varphi \otimes f \rangle \big).
\end{align*}
Rearranging the above and using Cauchy-Schwarz we obtain
\begin{equation*}
\int_{B_1} \varphi^k |df|^k \leq k \int_{B_1} \varphi^{k-1} |df|^{k-1} |d\varphi| |f|.
\end{equation*}
Applying Young's inequality to the preceding estimate, we obtain
\begin{equation*}
 \int_{B_1} \varphi^k |df|^k \leq \frac{1}{2} \int_{B_1} \varphi^k |df|^k + C_{k} \int_{B_1} |d\varphi|^k |f|^k
\end{equation*}
for some $C_k > 0$. We will denote by $C_k$ all further positive constants which depend only on $k$. Thus
\begin{equation} \label{eq:eq10}
 \int_{B_1} \varphi^k |df|^k \leq C_k \int_{B_1} |d\varphi|^k |f|^k,
\end{equation}
Now we have
\begin{align*}
\|df_0\|^k &\leq C_k \int_{B_{\frac{1}{2}}} |df|^k & & \text{(by Proposition~\ref{cor:mvt} with $r=\frac{1}{2}$)}\\
&\leq C_k \int_{B_1} \varphi^k |df|^k & & \text{(since $\varphi = 1$ on $B_{\frac{1}{2}}$}) \\
&\leq C_k \int_{B_1} |d\varphi|^k |f|^k & & \text{(by~\eqref{eq:eq10})} \\
&\leq C_k R^k |B_1| |d\varphi|_{\infty}^k & & \text{(since the image of $f$ is in $B_R$)} \\
&\leq C_k R^k, & &
\end{align*}
which yields~\eqref{claim:claim1}.
\end{proof}

\begin{cor} \label{cor:bounded-domains-R-hyperbolic}
Let $(U, g_0, \phi)$ be a calibrated manifold for which $U \subset \R^n$ is an open set, $g_0$ is the flat metric, and $\phi \in \Omega^k(\R^n)$ is a calibration.  If $U$ is bounded, then $U$ is $R_\phi$-hyperbolic.
\end{cor}
\begin{proof}
Since $U$ is bounded, we have $U \subset B_R \subset \mathbb{R}^n$ for some $R>0$. Take any point $p \in U$, vector $v \in T_pU$, and Smith immersion $f \colon B^k \rightarrow U$ having $f(0)=p$ and $df_0(e_1) = \frac{1}{a}v$, where $a > 0$.  Since $f$ is a Smith immersion, it is weakly conformal and $k$-harmonic, so Proposition~\ref{prop:bounded-k-harmonic-estimate} implies that $\|df_0\| \leq C_k R$ for some $C_k > 0$ depending only on $k$.  So, using that $f$ is weakly conformal, we have
\begin{equation*}
 \frac{1}{a} |v| = |df_0(e_1)| \leq \|df_0\| \leq C_kR,
\end{equation*}
which implies that $a \geq \frac{1}{C_k R}|v|$ and hence $K_U(v) \geq \frac{1}{C_k R}|v|$.  This shows that $U$ is $R_{\phi}$-hyperbolic.
\end{proof}

\begin{rmk}
Note that the above corollary is specific to flat euclidean space.  Indeed, the calibrated manifold $X = \CP^n$ equipped with the Fubini-Study metric and K\"{a}hler form $\phi = \omega_{\mathrm{FS}}$ is not $R_\phi$-hyperbolic, despite being compact (and hence bounded as a metric space).
\end{rmk}

\section{Spaces of continuous maps} \label{sub:SpacesCts}

\indent \indent We now digress to discuss some point-set topological aspects of families of continuous maps between manifolds.  The remarks in this section will be used in the following one. \\

\indent  Let $Z$ and $X$ be smooth manifolds, and let
$$C(Z,X) = \{ f \colon Z \to X \mid f \text{ is continuous}\},$$
equipped with the compact-open topology.  Since $Z$ and $X$ are manifolds, the topological space $C(Z,X)$ is second-countable and metrizable.  In particular, a subset $\mathcal{F} \subset C(Z,X)$ is compact if and only if it is sequentially compact.  \\
\indent Note that when $X$ is equipped with a topology-compatible metric (i.e., a metric that induces the given topology), the compact-open topology on $C(Z,X)$ coincides with the topology of compact convergence (i.e., uniform convergence on compact sets).  Consequently, if $d_1$ and $d_2$ are topology-compatible metrics on $X$, then a sequence in $C(Z,X)$ converges compactly with respect to $d_1$ if and only if it converges compactly with respect to $d_2$.  \\
\indent We are interested in necessary conditions for a family $\mathcal{F} \subset C(Z,X)$ of continuous maps to be pre-compact.  To be precise:

\begin{defn}
Let $\mathcal{F} \subset C(Z,X)$, where $Z$ and $X$ are smooth manifolds.
\begin{itemize}
\item Say $\mathcal{F}$ is \emph{pre-compact} if it has compact closure.  This is equivalent to saying that every sequence in $\mathcal{F}$ has a subsequence that converges to some $f \in C(Z,X)$.
\item Say $\mathcal{F}$ is \emph{pointwise equicontinuous with respect to $d_X$}, where $d_X$ is a topology-compatible metric on $X$, if for every $\varepsilon > 0$ and every $z_0 \in Z$, there exists a neighborhood $V \subset Z$ of $z_0$ such that for all $z \in V$ and $f \in \mathcal{F}$, we have $d_X( f(z), f(z_0)) < \varepsilon$.
\item Say $\mathcal{F}$ is an \emph{even family} if for any $z \in Z$, $x\in X$, and neighborhood $U \subset X$ of $x$, there exist a neighborhood $V \subset Z$ of $z$ and a neighborhood $W \subset U$ of $x$ such that any $f \in \mathcal{F}$ with $f(z) \in W$ satisfies $f(V) \subset U$.
\end{itemize}
\end{defn}

\begin{prop} \label{prop:EquiEven}
Let $\mathcal{F} \subset C(Z,X)$, where $Z$ and $X$ are smooth manifolds.  Let $d_X$ be a metric on $X$ that induces the topology.
\begin{enumerate}[(a)]
\item If $\mathcal{F}$ is pre-compact, then $\mathcal{F}$ is pointwise equicontinuous with respect to $d_X$.  If $X$ is compact, then the converse holds.
\item If $\mathcal{F}$ is pointwise equicontinuous with respect to $d_X$, then $\mathcal{F}$ is an even family.
\end{enumerate}
\end{prop}
\begin{proof} ${}$
\begin{enumerate}[(a)]
\item The first claim follows from the Arzel\`a–Ascoli theorem.  For the second claim, note that if $\mathcal{F}$ is pointwise equicontinuous with respect to $d_X$, then so too is its closure $\overline{\mathcal{F}}$.  The result now again follows from Arzel\`a–Ascoli.
\item This is a standard result (see, e.g., Kelley~\cite[page 237]{Kelley}), but we provide details here for completeness.  Let $z \in Z$, $x \in X$, and $U \subset X$ be a neighborhood of $x$, and choose $\varepsilon > 0$ such that $B_{\varepsilon}(x) \subset U$.  Since $\mathcal{F}$ is equicontinuous, there exists a neighborhood $V$ of $z$ that satisfies $f(V) \subset B_{\frac{\varepsilon}{2}}(f(z))$ for all $f \in \mathcal{F}$.  Setting $W = B_{\frac{\varepsilon}{2}}(x)$, we see that if $f \in \mathcal{F}$ satisfies $f(z) \in W$, then the triangle inequality gives $f(V) \subset U$. \qedhere
\end{enumerate}
\end{proof}

\indent Another necessary condition for pre-compactness is normality, which we now recall.

\begin{defn}
Let $\mathcal{F} \subset C(Z,X)$, where $Z$ and $X$ are smooth manifolds.
\begin{itemize}
\item A sequence $f_m \in \mathcal{F}$ is \emph{compactly divergent} if for any compact sets $K \subset Z$ and $L \subset X$, there exists $N \in \mathbb{N}$ such that $m \geq N$ implies $f_m(K) \cap L = \varnothing$.
\item Say $\mathcal{F}$ is a \emph{normal family} if every sequence $f_m \in \mathcal{F}$ admits a subsequence that is either convergent or compactly divergent.
\end{itemize}
\end{defn}

\noindent In practice, we shall use the following alternate characterization of normality.

\begin{prop}\label{prop:equivnormal} 
Let $\mathcal{F} \subset C(Z,X)$, where $Z$ and $X$ are smooth manifolds.  The following are equivalent:
\begin{enumerate}[(i)]
\item Every sequence $f_m \in \mathcal{F}$ admits a convergent subsequence or admits a compactly divergent subsequence. (That is, $\mathcal{F}$ is normal.)
\item Every sequence $f_m \in \mathcal{F}$ is either compactly divergent or admits a convergent subsequence.
\item For every sequence $f_m \in \mathcal{F}$, the following holds: If there exist compact sets $K \subset Z$ and $L \subset X$ and a subsequence $f_{m_j}$ with $f_{m_j}(K) \cap L \neq \varnothing$ for all $j$, then $f_m$ admits a convergent subsequence.
\end{enumerate}
\end{prop}
\begin{proof}  ${}$
\begin{itemize}
\item $(ii) \iff (iii)$.  Negating the definition of ``compactly divergent," we see that (iii) can be rewritten as follows: ``If $(f_m)$ is not compactly divergent, then $(f_m)$ admits a convergent subsequence." This is exactly (ii).

\item $(ii) \implies (i)$. This is immediate.

\item $(i) \implies (iii)$. Suppose (i) holds.  Let $f_m \in \mathcal{F}$ be a sequence for which there exist compact sets $K \subset Z$ and $L \subset X$ and a subsequence $(f_{m_j})$ such that $f_{m_j}(K) \cap L \neq \O$ for every $j \in \mathbb{N}$.  By hypothesis, $(f_{m_j})$ must itself admit a convergent subsequence or a compactly divergent subsequence.  But the latter is impossible, so $(f_{m_j})$ admits a convergent subsequence, thereby yielding a convergent subsequence of $(f_m)$.
 \qedhere
\end{itemize}
\end{proof}

\begin{rmk} Let $X^* = X \cup \{\infty\}$ denote the one-point compactification of $X$.  Since $X$ is a manifold, $X^*$ is second-countable and metrizable.  One can show that a sequence $f_j \in C(Z,X)$ is compactly divergent if and only if $f_j \to \infty$ in $C(Z,X^*)$.
\end{rmk}

\begin{prop} \label{prop:Equi-Normal}
Let $\mathcal{F} \subset C(Z,X)$, where $Z$ and $X$ are smooth manifolds.
\begin{enumerate}[(a)]
\item If $\mathcal{F}$ is pre-compact, then $\mathcal{F}$ is normal.  If $X$ is compact, then the converse holds.
\item Let $d_X$ be a topology-compatible metric on $X$.  Suppose that every closed, $d_X$-bounded subset of $X$ is compact.  If $\mathcal{F}$ is pointwise equicontinuous with respect to $d_X$, then $\mathcal{F}$ is normal.
\item If $\mathcal{F}$ is normal, then $\mathcal{F}$ is an even family.
\end{enumerate}
\end{prop}
\begin{proof} ${}$
\begin{enumerate}[(a)]
\item The first claim is immediate.  For the converse, note that if $X$ is compact, then a sequence in $\mathcal{F}$ cannot be compactly divergent.
\item This is proven in~\cite[Lemma 1.1]{Wu}, but we provide full details here.

We want to prove that $\mathcal{F}$ is normal, and we use characterization (iii) of Proposition~\ref{prop:equivnormal}.\\
Take any sequence $f_i \in \mathcal{F}$ such that there exist compact sets $K \subset Z$ and $K^{'} \subset X$ and a subsequence with $f_i(K) \cap K^{'}\neq \varnothing$, $\forall i$. Let $y_i \in f_i(K) \cap K'$, so that $y_i = f_i(x_i)$ for some $x_i \in K$. By compactness, up to passing to a further subsequence, we have $x_i \rightarrow x_0 \in K$, $y_i \rightarrow y_0 \in K^{'}$.

We claim that for any $z_0 \in Z$, there exists $R > 0$ such that
$$ A_{z_0} := \{f_i(z_0) \colon i \geq 1\} \subset \overline{B}_{R}^{d_X}(y_0), \quad \forall i. $$
The set $C \coloneqq \{z_0\} \cup\{x_i\}\cup\{x_0\}$ is compact, so it lies in some compact connected $C'\subset Z$. By equicontinuity, each $q\in C'$ has a neighborhood $V_q$ with $d_X(f(x),f(q))<\frac{1}{2}$ for all $x\in V_q$, $f\in \mathcal{F}$; hence any two points in $V_q$ have $d_X(f(x),f(x'))<1$.  By compactness, we may cover $C'$ by finitely many such $V_q$'s, which we call $U_1,\dots,U_N$.

Since $C'$ is connected and covered by $U_1,\dots,U_N$, any two points of $C'$ are joined by an overlapping chain of at most $N$ of these sets. Fix $i$ and a chain from $x_i$ to $z_0$. Choosing a point in each overlap gives points $v_0=x_i,\dots,v_{m+1}=z_0$ (with $m+1\le N$) with each consecutive pair in a common $U_k$, so $d_X(f_i(v_j),f_i(v_{j+1}))<1$. Summing, we obtain
\[d_X(f_i(z_0),f_i(x_i))<N.\]
With $M=\sup_i d_X(y_i,y_0)<\infty$, we have
\[
d_X(f_i(z_0),y_0)\le d_X(f_i(z_0),y_i)+d_X(y_i,y_0)<N+M.
\]
Thus $A_{z_0}\subseteq \overline{B}_{R}^{d_X}(y_0)$ with $R=N+M$.
We now claim that $f_i$ has a subsequence that converges uniformly (with respect to $d_X)$ on compact sets, which would complete the proof.  To see this, note that $f_i$ is pointwise equicontinuous.  Moreover, for each $z_0 \in Z$, the set $A_{z_0} \subset X$ is pre-compact in $X$.  Indeed, since its closure satisfies $\overline{A}_{z_0} \subseteq \overline{B}_R^{d_X}(y_0)$ by the above, it follows that $\overline{A}_{z_0}$ is $d_X$-bounded, and therefore compact by hypothesis.  Therefore, the conditions of the Arzel\`{a}-Ascoli theorem are fulfilled, so $f_i$ admits a subsequence that converges uniformly on compact subsets of $Z$ with respect to $d_X$.

\item For convenience, we equip $Z$ and $X$ with topology-compatible metrics. Assume for contradiction that $\mathcal{F}$ is not an even family. Then there exists $x_0 \in Z$, $p \in X$, and a neighborhood $U$ of $p$ in $X$ such that for all $i$, there exists $f_i \in \mathcal{F}$ satisfying
$$ f_i(x_0) \in B_{\frac{1}{i}}(p) \quad \text{but} \quad f_i(B_{\frac{1}{i}}(x_0)) \not\subset U, $$
where the second condition above says that there exist $x_i \in B_{\frac{1}{i}}(x_0)$ with $f_i(x_i) \notin U$. Note that $x_i \rightarrow x_0$.

Choose $i_0\in\mathbb N$ sufficiently large that $\overline{B}_{\frac{1}{i_0}}(p)$ is compact.
Set $K = \{x_0\} \subset Z$ and $K^{'} = \overline{B}_{\frac{1}{i_0}}(p) \subset X$, which are both compact. Then
$$ f_i(x_0) \in \overline{B}_{\frac{1}{i}}(p) \subset \overline{B}_{\frac{1}{i_0}}(p) \quad \text{for $i \geq i_0$}, $$
so $f_i(K) \cap K^{'} \neq \varnothing$.
 Now $\mathcal{F}$ is normal by the hypothesis, so there exists a subsequence $f_i$ that converges uniformly on compact subsets to $f_{\infty} \in C(Z,X)$.  We have the following facts:
\begin{itemize}
    \item from $f_i(x_0) \rightarrow p$ and uniform convergence, $f_{\infty}(x_0) = p$.
    \item $f_i \rightarrow f_{\infty}$ uniformly on compact sets, so $\mathrm{dist}_X(f_i(x_i), f_{\infty}(x_i)) \rightarrow 0$, because $\{x_i\} \cup \{x_0\}$ is compact.
    \item $f_{\infty}$ is continuous and $x_i \rightarrow x_0$, so $\mathrm{dist}_X(f_{\infty}(x_i), f_{\infty}(x_0)) \rightarrow 0.$
\end{itemize}
Combining all of these gives
$$ \mathrm{dist}_X(f_{i}(x_i),p) \leq \mathrm{dist}_X(f_i(x_i), f_{\infty}(x_i)) + \mathrm{dist}_X(f_{\infty}(x_i), f_{\infty}(x_0)) \rightarrow 0. $$
Since $U$ is an open neighborhood of $p$, eventually we have $f_i(x_i) \in U$, which is our desired contradiction. \qedhere
\end{enumerate}
\end{proof}

\begin{rmk} \label{rmk:CompactEquivalence}
From Propositions \ref{prop:EquiEven}(a) and \ref{prop:Equi-Normal}(a), we see that when $X$ is compact, for any metric $d_X$ inducing the topology of $X$, we have that:
$$\mathcal{F} \text{ is pre-compact} \ \iff \ \mathcal{F} \text{ is pointwise equicontinuous w.r.t. }d_X \ \iff \ \mathcal{F} \text{ is normal.}$$
\end{rmk}

\section{Royden's theorem and tautness in calibrated geometry} \label{sec:EquivHyp}

\indent \indent In $\S$\ref{sub:EquivHyp}, we prove Theorem \ref{thm:RoydenMain} (see Theorem \ref{theorem:theoremC}), a calibrated analogue of Royden's result.  Then, in $\S$\ref{sub:DiscreteQuot}, we apply this result to study the hyperbolicity and saturation of discrete quotients.  Finally, in $\S$\ref{sec:Tautness}, we prove Theorem \ref{thm:Kiernan-main} (see Theorem \ref{theorem:Royden}), a calibrated analogue of Kiernan's result.

\subsection{Equivalence of hyperbolicity notions} \label{sub:EquivHyp}

\indent \indent We will be concerned with $\phi$-saturated calibrated manifolds.  Our use of the saturation assumption will arise primarily through the following lemma and the subsequent theorem.

\begin{lem}\label{lemma:lemmaK}
Let $(X,g_X,\phi)$ be a calibrated manifold. The following are equivalent:
\begin{enumerate}[(i)]
\item $X$ is $\phi$-saturated.
\item For every $p\in X$, there exist a neighborhood $U\ni p$ and $C_U>0$ such that
\[
K_X(v)\leq C_U|v|\qquad(v\in TU).
\]
\item For every compact $K\subset X$, there exists $C_K>0$ such that
\[
K_X(v)\leq C_K|v|\qquad(v\in TX|_K).
\]
\end{enumerate}
\end{lem}
\begin{proof}
Assume (i) and fix $p\in X$. Since $K_X(0_p)=0$, there are a neighborhood $U\ni p$ and $\varepsilon>0$ such that
\[
q\in U,\quad |w_q|<\varepsilon\quad\Longrightarrow\quad K_X(w_q)<1.
\]
For $v_q\ne0$, set $w_q:=\frac{\varepsilon v_q}{2|v_q|}$. Positive homogeneity gives
\[
K_X(v_q)=\frac{2|v_q|}{\varepsilon}K_X(w_q)
\leq\frac{2}{\varepsilon}|v_q|.
\]
The inequality holds at $v_q=0$ as well, proving (ii).

Assume (ii). If $q_j\to p$, $v_j\in T_{q_j}X$, and $|v_j|_{g_X}\to0$, then eventually $q_j\in U$, and
\[
0\leq K_X(v_j)\leq C_U|v_j|\longrightarrow0.
\]
Thus (i) holds.

For (ii) implies (iii), cover a compact set $K$ by finitely many neighborhoods and take the maximum of their constants. Conversely, choose an open neighborhood with compact closure and apply (iii) to its closure to obtain (ii).
\end{proof}

\begin{thm}[$K_\phi$-hyperbolicity implies equivalence of $d_\phi$ and $\mathrm{dist}_g$ topologies] \label{thm:main-claim}
Let $(X,g,\phi)$ be $\phi$-saturated and $K_\phi$-hyperbolic (i.e., $d_{\phi}$ is non-degenerate).  Then the $d_\phi$-topology on $X$ coincides with the $\mathrm{dist}_g$-topology on $X$ (i.e., the given topology on $X$).
Moreover, $d_\phi$ is finite and $(X,d_\phi)$ is a length metric space.
\end{thm}
\begin{proof}
For every piecewise $C^1$ curve $\gamma$, Lemma~\ref{lemma:lemmaK}(iii) gives
\[
L_\phi(\gamma)\leq C_{\gamma([0,1])}\int_0^1|\gamma'(t)|\,dt<\infty.
\]
Since $X$ is connected, $d_\phi$ is finite. Take any $x \in X$ and choose $\delta > 0$ such that $B_{\delta}(x)$ is (strongly) geodesically convex and has compact closure with respect to the standard topology. Recall that this means that every pair of points in $B_{\delta}(x)$ can be joined by a minimizing geodesic inside $B_{\delta}(x)$.  We shall show that: (a) the $\mathrm{dist}_g$-topology is stronger than  the $d_\phi$-topology, and conversely (b) the $\mathrm{dist}_g$-topology is weaker than the $d_\phi$-topology.
\begin{enumerate}[(a)]
\item Set $K := \overline{B_{\delta}(x)}$, which is compact in the standard topology. Hence, by Lemma~\ref{lemma:lemmaK}, there exists $C_K > 0$ such that for all $y \in K$ and $v \in T_yX$ we have $K_X(v) \leq C_K |v|$. Let $\gamma \colon [0,1] \to X$ be a minimizing geodesic from $y$ to $z$ in $B_{\delta}(x)$. Then
\begin{equation} \label{eq:eq9}
 d_\phi(y,z) \leq \int_0^1{}^{*} K_X(\gamma'(t))\,dt \leq \int_0^1 C_K |\gamma'(t)|\,dt = C_K \mathrm{dist}_g(y,z),
\end{equation}
which means that the identity map $\mathrm{Id} \colon (X,\mathrm{dist}_g) \rightarrow (X,d_\phi)$ is locally Lipschitz, and thus continuous, yielding the required statement.

\item Consider the map $G \colon B_{\delta}(x) \rightarrow \mathbb{R}$ given by $G(y) = d_\phi(x,y)$.  The inequality~\eqref{eq:eq9} shows that $G$ is continuous with respect to the standard topology. So for any $\alpha \in (0,\delta)$, if we restrict $G$ to the compact set $\partial B_{\frac{\alpha}{2}}(x)$, then it attains a minimum, which is positive by the non-degeneracy of $d_\phi$. That is, there exists $y_0 \in \partial B_{\frac{\alpha}{2}}(x)$ such that $\displaystyle d_\phi(x,y_0) = \inf_{y\in \partial B_{\frac{\alpha}{2}}(x)} d_\phi (x,y) > 0$.

If $y \in X \setminus \overline{B_{\frac{\alpha}{2}}(x)}$, then any piecewise $C^1$ path $\gamma$ from $x$ to $y$ must pass through some point $y' \in \partial B_{\frac{\alpha}{2}}(x)$, so
$$ L_\phi(\gamma) \geq d_\phi (x, y') \geq d_\phi (x, y_0). $$
Taking infima over $\gamma$ gives $d_\phi(x,y)\geq d_\phi(x,y_0)$. It follows that
$$ \{y \in X : d_\phi(x,y) < d_\phi(x,y_0)\} \subset B_{\alpha}(x), $$
where $d_\phi(x,y_0) > 0$. This means that a base of open sets in the $\mathrm{dist}_g$-topology are also open in the $d_\phi$ topology, yielding the required statement.

\end{enumerate}
Since $L_\phi$ defines a length structure on $X$, $(X,d_\phi)$ is a length metric space~\cite{BBI}.
\end{proof}

We now prove the main result of this section.

\begin{thm}[Royden-type theorem] \label{theorem:theoremC}
Let $(X, g, \phi)$ be $\phi$-saturated. The following are equivalent:
\begin{enumerate}\setlength\itemsep{-0.5mm}
 \item[\emph{(1)}] $X$ is $R_{\phi}$-hyperbolic.
 \item[\emph{(2)}] $X$ is $K_{\phi}$-hyperbolic.
 \item[\emph{(3)}] $X$ is $K_{\phi}$-hyperbolic and $d_{\phi}$ induces the standard topology.
 \item[\emph{(4)}] $X$ is $K_\phi$-hyperbolic and $\mathrm{SmIm}(B^k,X)$ is pointwise equicontinuous with respect to $d_\phi$.
 \item[\emph{(5)}] $\mathrm{SmIm}(B^k,X)$ is an even family.
\end{enumerate}
Moreover, if $X$ is compact, then conditions (1) to (5) are equivalent to:
\begin{enumerate}
\item[\emph{(6)}] $\mathrm{SmIm}(B^k,X)$ is pre-compact. 
\end{enumerate}
\end{thm}
\begin{proof}
We will show that $(1) \implies (2) \implies (3) \implies (4) \implies (5) \implies (1)$.
\begin{itemize}
\item $(1) \implies (2)$. This is Theorem \ref{thm:R-Implies-K}.
\item $(2) \implies (3)$. This is  Theorem~\ref{thm:main-claim}.
\item $(3) \implies (4)$. In Proposition~\ref{prop:distance-properties} we showed that for every Smith immersion $f \colon (B^k, g_1, \vol_1) \to (X, g, \phi)$, we have $d_{\phi}(f(p), f(q)) \leq \mathrm{dist}_1(p,q)$ for all $p,q \in B^k$. This implies that $\mathrm{SmIm}(B^k, X)$ is uniformly equicontinuous with respect to $d_{\phi}$ and hence pointwise equicontinuous with respect to $d_{\phi}$.
\item $(4) \implies (5)$. By Theorem~\ref{thm:main-claim}, $d_\phi$ induces the standard topology. The result now follows from Proposition~\ref{prop:EquiEven}(b).

\item $(5) \implies (1)$. Take any $p \in X$. We need to find a neighborhood $W \subset X$ of $p$ and a constant $c > 0$, such that $K_X(v) \geq c |v|$ for all $v \in TW$.

Take $R_0>0$ and $C>0$ supplied by the Schwarz
Lemma~\ref{lem:lemJKS} at $p$, and fix $0<R\leq R_0$. Since $\mathrm{SmIm}(B^k,X)$ is an even family, there exist
$\delta\in(0,1)$ and $\delta'\in(0,R)$ such that every Smith immersion
$f\colon B^k\to X$ with $f(0)\in W:=B_{\delta'}(p)$ satisfies
$f(B_\delta(0))\subset B_R(p)$. By the Schwarz Lemma~\ref{lem:lemJKS}, every Smith immersion $f \colon B^k \to X$ with $f(0) \in B_{\delta'}(p)$ satisfies
\begin{equation} \label{eq:eq11}
\| df_0 \| \leq C \frac{R}{\delta}.
\end{equation}
Hence, for any $v \in TW$, taking any Smith immersion $f \colon B^k \rightarrow X$ with $f(0) \in W = B_{\delta'}(p)$ and $df_0(e_1) = \frac{1}{a} v$, and thus using~\eqref{eq:eq11} we get
\begin{align*}
 \frac{1}{a} |v| = |df_0 (e_1)| \leq \|df_0\| \leq C\frac{R}{\delta}.
\end{align*}
Rearranging gives $a \geq \frac{\delta}{CR} |v|$, and then taking infima over all such $f$ yields the required inequality $K_X(v) \geq \frac{\delta}{CR} |v|$ for all $v \in TW$.
\end{itemize}
Finally, suppose $X$ is compact.  By Theorem~\ref{thm:main-claim} and Remark \ref{rmk:CompactEquivalence}, we see that (4) $\implies$ (6).  By Proposition \ref{prop:EquiEven}, we have (6) $\implies$ (5).
\end{proof}

\begin{rmk}\label{rmk:EvenRoyden}
In the previous theorem, our proof of (5) $\implies$ (1) does not use the $\phi$-saturation hypothesis.
\end{rmk}

\subsection{Hyperbolicity of discrete quotients} \label{sub:DiscreteQuot}

\indent \indent We now turn to the hyperbolicity of discrete quotients.  Throughout this section, we work in the following setting.  Let $(\tilde{X},g_{\tilde{X}})$ be an oriented Riemannian manifold, and let $\widetilde{\phi} \in \Omega^k({\tilde{X}})$ be a $\Gamma$-invariant calibration, where $\Gamma \leq \mathrm{Isom}(g_{\tilde{X}})$ is a discrete subgroup that acts freely and properly on ${\tilde{X}}$. Let $\pi \colon{\tilde{X}} \rightarrow {\tilde{X}}/\Gamma =: X$ be the projection. Equip $X$ with the unique Riemannian metric $g_X$ such that $\pi \colon ({\tilde{X}},g_{\tilde{X}}) \rightarrow (X,g_X)$ is a local isometry, and with the unique calibration $\phi \in \Omega^k(X)$ that satisfies $\pi^*\phi = \widetilde{\phi}$.

\begin{lem}[Lifts of Smith immersions] \label{lem:Lifting}
Let $(\Sigma, g_\Sigma)$ be an oriented Riemannian $k$-manifold.  If $\widetilde{f} \colon \Sigma \to \widetilde{X}$ is a Smith immersion, then $\pi \circ \widetilde{f} \colon \Sigma \to X$ is a Smith immersion.  Conversely, if $f \colon \Sigma \to X$ is a Smith immersion and $\Sigma$ is simply-connected, then there exists a lift $\widetilde{f} \colon \Sigma \to \tilde{X}$ of $f$ that is a Smith immersion.
\end{lem}
\begin{proof}
Suppose that we have the following commutative diagram:
\[
\begin{tikzcd} & \tilde{X} \arrow{d}[swap]{\pi} \\
 \Sigma \arrow{r}[swap]{f} \arrow{ur}{\tilde{f}} & X 
\end{tikzcd}
\]
We claim that $f$ is a Smith immersion if and only if $\tilde{f}$ is a Smith immersion. To see this, from the two equalities
\begin{align*}
 f^{*} \phi = \tilde{f}^{*} \pi^{*} \phi= \tilde{f}^{*} \widetilde{\phi},
\end{align*}
and
\begin{align*}
g_X(df,df) = g_X (d\pi \circ d\tilde{f},d\pi \circ d\tilde{f}) = \pi^* g_X (d\tilde{f},d\tilde{f}) = g_{\tilde{X}}(d\tilde{f},d\tilde{f}),
\end{align*}
we immediately get that
$$ f^{*} \phi =\frac{g_X(df,df)^{\frac{k}{2}}}{(\sqrt{k})^k} \mathrm{vol}_{\Sigma} \, \iff \, \tilde{f}^{*} \widetilde{\phi} =\frac{g_{\tilde{X}}(d\tilde{f},d\tilde{f})^{\frac{k}{2}}}{(\sqrt{k})^k} \mathrm{vol}_\Sigma. $$
In particular, if $\tilde{f} \colon \Sigma \to \tilde{X}$ is a Smith immersion, then $f = \pi \circ \tilde{f} \colon \Sigma \to X$ is a Smith immersion.  Conversely, if $f \colon \Sigma \to X$ is a Smith immersion, then since $\Sigma$ is simply-connected, we have $f_{*}(\pi_1 (\Sigma)) = 0 \subset \pi_{*}(\pi_1(\tilde{X}))$, and thus there exists a lift $\tilde{f} \colon \Sigma \rightarrow \tilde{X}$, and this lift is necessarily also a Smith immersion.
\end{proof}

\begin{prop}[Hyperbolicity of discrete quotients] \label{prop:DiscreteQuot}
In the situation described above, the following results hold.
\begin{enumerate}[(a)] \setlength\itemsep{-0.2mm}
 \item ${\tilde{X}}$ is $\widetilde{\phi}$-hyperbolic $\, \iff \,$ $X$ is $\phi$-hyperbolic.
 \item ${\tilde{X}}$ is $R_{\widetilde{\phi}}$-hyperbolic  $\, \iff \,$  $X$ is $R_{\phi}$-hyperbolic.
 \item ${\tilde{X}}$ is $\widetilde{\phi}$-saturated  $\, \iff \,$  $X$ is $\phi$-saturated.
 \item Suppose that $\tilde{X}$ is $\widetilde{\phi}$-saturated (or equivalently, by (c), that $X$ is $\phi$-saturated). Then ${\tilde{X}}$ is $K_{\widetilde{\phi}}$-hyperbolic  $\, \iff \,$  $X$ is $K_{\phi}$-hyperbolic.
\end{enumerate}
\end{prop}
\begin{proof} ${}$
\begin{enumerate}[(a)]
\item Suppose first $X$ is $\phi$-hyperbolic. Take any Smith immersion $\tilde{f} \colon \mathbb{R}^k \rightarrow \tilde{X}$. Then $\pi \circ \tilde{f} \colon \mathbb{R}^k \to X$ is again a Smith immersion by Lemma \ref{lem:Lifting}. By hypothesis, $\pi \circ \tilde{f}$ is constant, so there exists $p\in X$ such that $(\pi \circ \tilde{f}) (x) = p$ for all $x \in \mathbb{R}^k$.  That is, $\tilde{f}(\mathbb{R}^k) \subseteq \pi^{-1}(p)$. But the latter is a discrete set, so the continuity of $\tilde{f}$ implies that $\tilde{f}$ is constant.  Thus, $\tilde{X}$ is $\widetilde{\phi}$-hyperbolic.

For the converse, suppose $\tilde{X}$ is $\widetilde{\phi}$-hyperbolic.  Let $f \colon \mathbb{R}^k \rightarrow X$ be a Smith immersion. By Lemma \ref{lem:Lifting}, there exists a lift $\tilde{f} \colon \mathbb{R}^k \rightarrow \tilde{X}$, which is a Smith immersion. By hypothesis, $\tilde{f}$ is constant, and hence so is $f = \pi \circ \tilde{f}$, and thus $X$ is $\phi$-hyperbolic.
\item First, we establish a property which we will use in the remainder of the proof:
\begin{equation} \label{eq:eq6}
 K_{\tilde{X}}(v) = K_X(d\pi (v)), \text{ for all $v\in T\tilde{X}$.}
\end{equation}
This is clear since for any Smith immersion $\tilde{f} \colon B^k \rightarrow \tilde{X}$ having $\tilde{f}(0)=\tilde{p}$ and $d\tilde{f}_0(e_1)=\frac{1}{a}v$ for $a > 0$, Lemma \ref{lem:Lifting} yields a Smith immersion $f \colon B^k \rightarrow X$ having $f(0)=\pi(\tilde{p})$ and $df_0(e_1)=\frac{1}{a}d\pi_{\tilde{p}} (v)$, and vice-versa.

Since any vector $w \in TX$ is of the form $w = d\pi(v)$ for some $v \in T\tilde{X}$, and since $\pi$ is a local isometry, we immediately get that ${\tilde{X}}$ is $R_{\widetilde{\phi}}$-hyperbolic if and only if $X$ is $R_{\phi}$-hyperbolic.
\item Observe that~\eqref{eq:eq6} implies that $K_{\tilde{X}}$ is finite if and only if $K_X$ is finite.

For upper semi-continuity at the zero section, suppose first that $X$ is $\phi$-saturated. Let $\tilde q_j\to\tilde p$, $\tilde v_j\in T_{\tilde q_j}\tilde X$, and $|\tilde v_j|_{g_{\tilde X}}\to0$. Since $\pi$ is a local isometry,
\[
\pi(\tilde q_j)\to\pi(\tilde p),\qquad
|d\pi(\tilde v_j)|_{g_X}=|\tilde v_j|_{g_{\tilde X}}\to0.
\]
Thus~\eqref{eq:eq6} and saturation give
\[
\limsup_{j\to\infty}K_{\tilde X}(\tilde v_j)
=\limsup_{j\to\infty}K_X(d\pi(\tilde v_j))=0.
\]
For the other direction, suppose that $\tilde X$ is $\widetilde\phi$-saturated. Let $q_j\to p$, $v_j\in T_{q_j}X$, and $|v_j|_{g_X}\to0$. Choose $\tilde p\in\pi^{-1}(p)$ and a local inverse $s\colon U\to\tilde X$ of $\pi$, where $U$ is an open neighborhood of $p$ and $s(p)=\tilde p$. For all sufficiently large $j$, we have $q_j\in U$, so set $\tilde q_j:=s(q_j)$ and $\tilde v_j:=(ds)_{q_j}(v_j)$. Then
\[
\tilde q_j\to\tilde p,\qquad
|\tilde v_j|_{g_{\tilde X}}=|v_j|_{g_X}\to0,\qquad
d\pi(\tilde v_j)=v_j.
\]
Hence~\eqref{eq:eq6} and saturation give
\[
\limsup_{j\to\infty}K_X(v_j)
=\limsup_{j\to\infty}K_{\tilde X}(\tilde v_j)=0.
\]

\item For the first direction, assume $\tilde{X}$ is $K_{\widetilde{\phi}}$-hyperbolic, that is $d_{\widetilde{\phi}}$ is non-degenerate. Take any $p,q \in X$ such that $d_\phi(p,q)=0$. That means that for every $\varepsilon > 0$ there exists a piecewise $C^1$ curve $\gamma \colon [0,1] \rightarrow X$ from $p$ to $q$ with $\int_0^1{}^{*} K_X(\gamma'(t))\,dt < \varepsilon$. Fix $\tilde{p} \in \pi^{-1}(p)$. Then take the unique lift $\tilde{\gamma} \colon [0,1] \rightarrow \tilde{X}$ of $\gamma$ such that $\tilde{\gamma}(0)=\tilde{p}$ and set $\tilde{q} := \tilde{\gamma}(1) \in \pi^{-1}(q)$. Then using~\eqref{eq:eq6}, we get
\begin{align*}
 \int_0^1{}^{*} K_{\tilde{X}}(\tilde{\gamma}^{'}(t)) dt= \int_0^1{}^{*} K_{{X}}({\gamma}^{'}(t)) dt< \varepsilon,
\end{align*}
implying that $$d_{\widetilde{\phi}}(\tilde{p}, \pi^{-1}(q)) = 0. $$
But $\pi^{-1}(q)$ is closed with respect to the standard topology and hence by Theorem~\ref{thm:main-claim}, closed with respect to the $d_{\widetilde{\phi}}$-topology. This means that $\tilde{p} \in \pi^{-1}(q)$ implying that $p = \pi(\tilde{p}) = q$, concluding the first direction.

For the other direction, assume that $d_\phi$ is non-degenerate. Take any $\tilde{p}, \tilde{q} \in \tilde{X}$ such that $d_{\widetilde{\phi}}(\tilde{p},\tilde{q}) =0$. Let $p := \pi(\tilde{p})$, $q := \pi(\tilde{q})$. Take any piecewise $C^1$ curve $\tilde{\gamma}$ on $\tilde{X}$ between $\tilde{p}, \tilde{q}$. Then $\gamma := \pi\circ \tilde{\gamma}$ is a curve on $X$ between $p,q$. Then using~\eqref{eq:eq6} again we get
\begin{equation*}
 d_\phi(p,q) \leq \int_0^1{}^{*} K_{{X}}({\gamma}^{'}(t))dt = \int_0^1{}^{*} K_{\tilde{X}}(\tilde{\gamma}^{'}(t))dt.
\end{equation*}
Taking infima over all such curves $\tilde{\gamma}$, we conclude:
\begin{equation*}
 d_\phi(p,q) \leq d_{\widetilde{\phi}}(\tilde{p}, \tilde{q}).
\end{equation*}
By assumption $d_{\widetilde{\phi}}(\tilde{p}, \tilde{q}) = 0$ and hence, $ d_\phi(p,q) = 0$ which implies $p=q$. Thus, $\pi(\tilde{p}) = \pi(\tilde{q})$ meaning that $\tilde{p},\tilde{q} \in \pi^{-1}(p)$. It remains to show that $\tilde{p} = \tilde{q}$. Assume for contradiction that $\tilde{p} \neq {\tilde{q}}$. Since $\pi$ is a local diffeomorphism, there exists a relatively compact open ball $B_{\varepsilon}(\tilde{p})$ of radius $\varepsilon$ around $\tilde{p}$ with respect to the standard topology such that $\pi$ is a diffeomorphism on a neighborhood of $\overline{B_{\varepsilon}(\tilde p)}$ and $B_{\varepsilon}(\tilde{p}) \cap (\pi^{-1}(p) \setminus \{\tilde{p}\}) = \varnothing$. Since $\pi$ is a local diffeomorphism, it is an open map, so consider the open set $U := \pi(B_{\varepsilon}(\tilde{p}))$ in $X$. Clearly, $p \in U$, and we have $\pi(\partial B_{\varepsilon}(\tilde p))=\partial U$. Since $d_\phi$ is non-degenerate, by Theorem~\ref{thm:main-claim}, the standard topology on $X$ coincides with the topology induced by $d_\phi$ and thus there exists an open ball $B_{\delta}(p) \subset U$ around $p$ of radius $\delta$ with respect to the $d_\phi$ topology in $X$. Now, every piecewise $C^1$ curve $\tilde{\gamma}$ from $\tilde p$ to $\tilde q$ has to pass through the boundary of $B_{\varepsilon}(\tilde{p})$, which implies that $\gamma = \pi \circ \tilde{\gamma}$ has to pass through the boundary of $U$ and hence through the boundary of $B_{\delta}(p)$. We deduce that
\begin{equation*}
 \int_0^1{}^{*} K_{\tilde{X}}(\tilde{\gamma}^{'}(t)) dt = \int_0^1{}^{*} K_{{X}}(\gamma'(t)) dt \geq d_\phi(p,\partial B_{\delta}(p)) = \delta.
\end{equation*}

Taking infima over all curves $\tilde{\gamma}$, we get
\begin{equation*}
d_{\widetilde{\phi}}(\tilde{p},\tilde{q}) \geq \delta,
\end{equation*}
where $\delta > 0$ is a fixed constant, which is our desired contradiction. \qedhere
\end{enumerate}

\end{proof}

\begin{example}
By means of Proposition \ref{prop:DiscreteQuot}(c), Corollary~\ref{prop:flat-saturation}, and~\cite[Example 6.7]{BIM}, we find the following non-K\"{a}hler examples of \emph{compact} saturated calibrated manifolds.
\begin{itemize}
\item Let $\widetilde{\phi} \in \Omega^k(\R^n)$ be a constant-coefficient elliptic calibration.  Then $T^n = \R^n/\Z^n$ is $\phi$-saturated, where $\phi$ is the calibration on $T^n$ having $\pi^*\phi = \widetilde{\phi}$ for the projection $\pi \colon \R^n \to \R^n/\Z^n = T^n$.  Note that $T^n$ is not $\phi$-hyperbolic.
\item The compact quaternionic-K\"{a}hler manifold $X = \mathbb{HH}^n/\Gamma$ is compact and $\Psi$-saturated, where $\Gamma \leq \mathrm{Sp}(n,1)$ is a cocompact lattice acting freely on $\mathbb{HH}^n$, and $\Psi \in \Omega^4(X)$ is the QK calibration.  Note that $X$ is $R_\Psi$-hyperbolic (and hence $\Psi$-hyperbolic).
\end{itemize}
\end{example}

\subsection{Tautness in calibrated geometry} \label{sec:Tautness}

\indent \indent As discussed in $\S$\ref{sub:MainResults}, a complex manifold is \emph{taut} if $\mathrm{Hol}(B^2, X)$ is a normal family.  The term ``taut" is due to Wu~\cite{Wu}, who was among the first to consider normal families of holomorphic maps in higher dimensions.  By analogy, we make the following definition.

\begin{defn}
A calibrated manifold $(X, g, \phi)$ is called \emph{$\phi$-taut} if $\mathrm{SmIm}(B^k, X)$ is a normal family.
\end{defn}

\begin{example} If $(X, g_X, \omega)$ is a K\"{a}hler manifold, then $\mathrm{SmIm}(B^2, X) = \mathrm{Hol}(B^2, X)$. Consequently, $X$ is $\omega$-taut if and only if  $X$ is taut as a complex manifold.
\end{example}

\indent For complex manifolds, Kiernan~\cite{Kiernan} proved that tautness is an intermediate notion between Kobayashi hyperbolicity and complete Kobayashi hyperbolicity.  As we now show, the same is true in the calibrated setting as well.

\begin{thm}[Kiernan-type theorem] \label{theorem:Royden}
Let $(X, g, \phi)$ be a calibrated manifold.
\begin{enumerate}[(a)]\setlength\itemsep{-0.2mm}
\item If $X$ is $\phi$-saturated, $K_{\phi}$-hyperbolic and $(X, d_{\phi})$ is complete, then $X$ is $\phi$-taut.
\item If $X$ is $\phi$-taut, then  $\mathrm{SmIm}(B^k,X)$ is an even family, and hence $X$ is $R_\phi$-hyperbolic.
\end{enumerate}
\end{thm}
\begin{proof} ${}$
\begin{enumerate}[(a)] \setlength\itemsep{-0.2mm}
\item Suppose $X$ is $\phi$-saturated, $K_\phi$-hyperbolic, and $(X, d_\phi)$ is complete. By Theorem ~\ref{thm:main-claim}, the distance function $d_\phi$ induces the topology of $X$.  Since $(X, d_\phi)$ is a length metric space, the Hopf-Rinow theorem for locally compact length metric spaces implies that every closed, $d_\phi$-bounded subset of $X$ is compact. By the implication (2)  $\implies$ (4) of Theorem \ref{theorem:theoremC}, we have that $\mathrm{SmIm}(B^k, X)$ is pointwise equicontinuous with respect to $d_\phi$.  Therefore, Proposition \ref{prop:Equi-Normal}(b) implies that $\mathrm{SmIm}(B^k, X)$ is a normal family, so $X$ is $\phi$-taut.

\item If $X$ is $\phi$-taut, then $\mathrm{SmIm}(B^k,X)$ is a normal family, so the result follows directly from Proposition~\ref{prop:Equi-Normal}(c) and Remark \ref{rmk:EvenRoyden}. \qedhere
\end{enumerate}
\end{proof}

\begin{cor}\label{cor:compact-equivalences}
Let $(X, g, \phi)$ be a compact, $\phi$-saturated calibrated manifold.  Then conditions (1) through (6) in Theorem \ref{theorem:theoremC} are equivalent to:
\begin{enumerate}\setlength\itemsep{-0.5mm}
 \item[\emph{(7)}] $X$ is $\phi$-taut.
\end{enumerate}
\end{cor}

\begin{proof}
Since $X$ is compact, Proposition~\ref{prop:Equi-Normal}(a) gives (6) $\iff$ (7).
\end{proof}

\bibliographystyle{plain}
\bibliography{MontelstheoremBib}

\Addresses

\end{document}